\documentclass{amsart}

\usepackage{cite}
\usepackage[all]{xy}
\usepackage{graphicx}
\usepackage{textcomp}
\usepackage{charter}
\usepackage[vflt]{floatflt}
\usepackage{amssymb}

\newtheorem{thm}{Theorem}[section]
\newtheorem{prop}[thm]{Proposition}
\newtheorem{lem}[thm]{Lemma}
\newtheorem{cor}[thm]{Corollary}
\newtheorem{rem}[thm]{Remark}
\newtheorem{ex}[thm]{Example}
\newtheorem{quest}[thm]{Question}

\theoremstyle{definition}
\newtheorem{definition}[thm]{Definition}

\theoremstyle{remark}

\numberwithin{equation}{section}

\newcommand{\GL}{\operatorname{GL}}
\newcommand{\SL}{\operatorname{SL}}

\newcommand{\bu}{{\mathbf u}}

\newcommand{\buc}{\bu^{\times}_D}
\newcommand{\bucl}{\bu^{\times}_{D,\ell}}

\newcommand{\cH}{{\mathcal H}}

\newcommand{\cL}{{\mathcal L}}

\newcommand{\cC}{{\mathcal C}}

\newcommand{\cO}{{\mathcal O}}
\newcommand{\cF}{{\mathcal F}}

\newcommand{\cK}{{\mathcal K}}

\newcommand{\C}{{\mathbb C}}
\newcommand{\R}{{\mathbb R}}
\newcommand{\pp}{\mathbb{P}}

\newcommand{\Z}{{\mathbb Z}}
\newcommand{\T}{{\mathbb T}}

\newcommand{\Hom}{\operatorname{Hom}}

\newcommand{\Span}{\operatorname{Span}}
\newcommand{\Proj}{\operatorname{Proj}}

\newcommand{\Gr}{\mathbb{G}}
\newcommand{\codim}{\operatorname{codim}}

\newcommand{\reas}{reasonable}

\newcommand{\vr}{very reasonable}

\begin{document}

\title{Livsic-type determinantal representations and Hyperbolicity}


\author{E. Shamovich}
\address{Department of Mathematics, Ben-Gurion University of the Negev}
\email{shamovic@math.bgu.ac.il}
\thanks{The research of E. S. was partially carried out during the visits to the of Mathematics and Statistics of the University of Konstanz, supported by the EDEN Erasmus Mundus program (30.12.2013 - 30.6.2014) and to the MFO, supported by the Leibnitz graduate student program (6.4.2014-12.4.2014). The research of E. S. was also supported by the Negev fellowship of the Kreitman school of the Ben Gurion University of the Negev.}

\author{V. Vinnikov}
\address{Department of Mathematics, Ben-Gurion University of the Negev}
\email{vinnikov@math.bgu.ac.il}

\thanks{Both authors were partially supported by
US--Israel  BSF grant 2010432.}

\begin{abstract}
Hyperbolic homogeneous polynomials with real coefficients, i.e.,
hyperbolic real projective hypersurfaces, and their determinantal representations,
play a key role in the emerging field of convex algebraic geometry.
In this paper we consider a natural notion of hyperbolicity for a real
subvariety $X \subset \pp^d$ of an arbitrary codimension $\ell$ with respect to
a real $\ell - 1$-dimensional linear subspace $V \subset \pp^d$
and study its basic properties.
We also consider a special kind of determinantal representations
that we call Livsic-type and a nice subclass of these that we call \vr{}.
Much like in the case of hypersurfaces ($\ell=1$), the existence
of a definite Hermitian \vr{} Livsic-type determinantal representation implies hyperbolicity.
We show that every curve admits a \vr{} Livsic-type determinantal representation.
Our basic tools are Cauchy kernels for line bundles and the notion of the Bezoutian for two meromorphic functions
on a compact Riemann surface that we introduce.
We then proceed to show that every real curve in $\pp^d$ hyperbolic
with respect to some real $d-2$-dimensional linear subspace admits
a definite Hermitian, or even real symmetric, \vr{} Livsic-type determinantal representation.
\end{abstract}

\maketitle


\section{Introduction}

The study of hyperbolic polynomials originated with the theory of partial differential equations.
A linear partial differential equation with constant coefficients is called hyperbolic
if there exists $a \in \pp^d(\R)$ such that the symbol $p$, considered as a homogeneous polynomial,
satisfies $p(a) \neq 0$ and $p(a + tx) = 0$ only if $t \in \R$ for every $x \in \pp^d(\R)$.
This led G\"{a}rding \cite{Ga51,Ga59} and Lax \cite{Lax58} to consider such polynomials
and the hypersurfaces $X(\R)=\left\{x \in \pp^d(\R) \colon p(x)=0\right\}$ they define.
In particular, G\"{a}rding proved in \cite{Ga59}
that if $p$ is hyperbolic with respect to $a$ as above
then the connected component $C$ of $a$ in $\pp^d(\R) \setminus X(\R)$ is convex and $p$
is hyperbolic with respect to any $a'$ in $C$
(in the case when $X$ is irreducible or $X(\R)$ is smooth,
$C$ simply consists of all $a' \in \pp^d(\R)$ such that $p$
is hyperbolic with respect to $a'$).
More precisely,
the cone over the set $C$ in $\R^{d + 1}$ has two connected components, each one a convex cone.
During the last two decades these hyperbolicity cones came to play an important role in optimization
and related fields \cite{Gu97,BGLS01,Re06}.
Among other applications, hyperbolic polynomials played a key role
in the recent proof by Marcus, Spielman and Srivastava
of the Kadison--Singer conjecture in operator algebras \cite{MSS14}.

A simple way to manufacture hyperbolic polynomials is to consider Hermitian matrices $A_0,\ldots A_d$
such that $A_0 > 0$,
and set $p(x_0,\ldots,x_d) = \det \left(\sum_{j=0}^d x_j A_j \right)$.
Then since $A_0 > 0$, we see easily (using the fact the eigenvalues of a Hermitian matrix are real)
that $p$ is hyperbolic with respect to $(1:0:\ldots :0)$.
Furthermore, the connected component of $(1,0,\ldots,0)$
in $\{x \in \R^{d+1} \colon p(x) \neq 0\}$ is given
by the linear matrix inequality $\sum_{j=0}^d x_j A_j > 0$,
i.e., the hyperbolicity cone is a spectrahedral cone \cite{RG95}
which is the feasible set of a semidefinite program,
see \cite{NN94,VB96,Nem06} as well as
the recent survey volume \cite{CA G}.
In this case we say that $p$ admits a definite Hermitian determinantal representation.

Using the correspondence between determinantal representations and kernel line bundles \cite{Vin89}
that goes in its essence back to Dixon \cite{D1900},
and a detailed analysis of the real structure of the corresponding Jacobian variety,
it was shown by the second author in \cite{Vin93} that
for a smooth real hyperbolic curve in $\pp^2$,
definite determinantal representations are parametrized by points on a certain distinguished real torus
in the Jacobian.
In particular, every smooth real hyperbolic curve in $\pp^2$ admits a definite determinantal representation,
a fact established previously by Dubrovin \cite{Dub83}.
A technique using the Cauchy kernels for vector bundles was developed in \cite{BV-ZPF}
(following \cite{BV-ZPT})
to provide a construction of determinantal representations for any plane algebraic curve.
This technique was later used by Helton and the second author in \cite{HelVin07}
to prove that every real hyperbolic plane curve admits a definite Hermitian and
even a real symmetric determinantal representation, settling a conjecture of Lax \cite{Lax58}.
(The result in \cite{HelVin07} is in the nonhomogeneous setting of real zero polynomials ---
the explicit translation to the homogeneous setting of hyperbolic polynomials and the connection
to the Lax conjecture were worked out in \cite{LPR05}.)

If we consider hypersurfaces in $\pp^d$ for $d > 2$,
we immediately see by a count of parameters argument \cite{Dic21} or by a Bertini theorem argument as in \cite{Bea00}
that a generic hypersurface does not admit a determinantal representation (except for quadrics and cubics in $\pp^3$).
Determinantal representations of possibly singular and multiple hypersurfaces in $\pp^d$
were considered in details by Kerner and the second author in \cite{KerVin12} to which we also refer for further references.
It was proved by Branden in \cite{Bra11}
that even if we allow multiplicity structure not every real hyperbolic hypersurface
will admit a definite determinantal representation.
We refer to \cite{Vppf} for an up-to-date survey on definite determinantal representations of real hyperbolic hypersurfaces
and linear matrix inequality representations of the corresponding hyperbolicity cones;
see also \cite{Kum} for a recent progress.

In this paper we proceed in a different direction:
we consider determinantal representations and hyperbolicity for subvarieties $X \subset \pp^d$ ($d \geq 2$)
of an arbitrary codimension $\ell \geq 1$,
both in general and in the case of curves.

In Section \ref{sec:det_rep} we define a special kind of determinantal representations that
we call Livsic-type determinantal representations
that generalize both linear determinantal representations of hypersurfaces and
the determinantal representations of curves considered in \cite{LKMV} in the context of multivariable operator theory
and multidimensional systems (vessels).
We then show that a specific subclass of Livsic-type determinantal representations,
that we call \vr{}, has especially nice properties.
In particular, if $X$ admits a \vr{} Livsic-type determinantal representation,
then the associated hypersurface $Y$ in the Grassmanian $\Gr(\ell - 1,d)$
of $\ell-1$-dimensional linear subspaces of $\pp^d$
(that consists of linear subspaces that intersect $X$)
admits a linear determinantal representation.

In Section \ref{sec:hyper} we define the notion of hyperbolicity
for subvarieties of $\pp^d$ of an arbitrary codimension:
we call a real subvariety $X$ hyperbolic with respect to a real linear subspace $V \subset \pp^d$ of dimension $\ell-1$
if $X \cap V = \emptyset$ and for every real linear subspace $U \subset \pp^d$ of dimension $\ell$ containing $V$,
$X \cap U$ consists of only real points.
Equivalently, every real $1$-dimensional Schubert cycle through $V$ in the Grassmanian
intersects the associated hypersurface $Y$ in real points only.
We show that the connected component $C(V)$ of $V$ in $\Gr(\ell - 1,d)(\R) \setminus Y(\R)$
has a natural convexity property that we call slice-convexity,
and that $X$ is hyperbolic with respect to any $V' \in C(V)$.
It is an open question whether $C(V)$ (more precisely any of the two connected components of the cone over it in the Pl\"{u}cker embedding) has a different property of being extendably convex in the sense of Buseman \cite{Bus61}
(an intersection of a convex set in the ambient space with the image of the Grassmanian), or whether, in the case when $X$ is irreducible or when $X(\R)$ is smooth,
$C(V)$ coincides with the set of all $\ell-1$-dimensional real linear subspaces $V'$ so that
$X$ is hyperbolic with respect to $V'$.

We also demonstrate that if $X$ admits a very reasonable Livsic-type determinantal representations
that is definite Hermitian, then $X$ is hyperbolic.

Sections \ref{sec:bezoutian}--\ref{sec:hyper_curve} are dedicated to Livsic-type determinantal representations
and hyperbolicity for curves in $\pp^d$.
While our methods are a generalization of the methods used in \cite{BV-ZPF} and \cite{HelVin07},
it is both more natural and more convenient to set them in the framework of Bezoutians
on a compact Riemann surface.

In Section \ref{sec:bezoutian} we introduce the notion of a Bezoutian of two meromorphic functions
with simple poles on a compact Riemann surface;
this notion originated in the study of Hankel-type realizations for meromorphic bundle maps
on a compact Riemann surface as transfer functions of overdetermined 2D systems (vessels) \cite{bvh},
and seems to be appropriate for studying localization of zeroes just as in the classical (genus zero) case. Similar notions of resultants of meromorphic functions on a Riemann surface were considered by Gustafsson and Tkachev in \cite{GuTk09} and \cite{GuTk11}.
We limit ourselves to proving several basic properties of the Bezoutian
that are essential for our purposes here, and postpone a more general development of the theory and applications
(as well as clarifying the relation to the work of Shapiro and the second author \cite{AS,ASV1,ASV2})
to a future publication.
In Section \ref{sec:div_func} we consider Bezoutians on compact real Riemann surfaces
(a Riemann surface equipped with an antiholomorphic involution $\tau$ or equivalently the desingularization
of a real algebraic curve)
and in particular on those of dividing type. We show how the Bezoutian relates to dividing functions,
i.e., real meromorphic functions that map a half of the compact real Riemann surfaces of dividing type
onto the upper half plane
and that are closely related to the hyperbolicity of the Riemann surface birationally embedded as an algebraic
curve in a projective space.

In Section \ref{sec:det_rep_curve}
we use the Bezoutians to show that every curve $X \subset \pp^d$ admits a \vr{} Livsic-type determinantal representations,
generalizing the construction of \cite{BV-ZPF} in the case $d=2$
and (essentially) the construction of Kravitsky \cite{Kra} (see also \cite{LKMV})
in the case of rational curves (genus zero).
Finally, in Section \ref{sec:hyper_curve} we extend the results of \cite{HelVin07} in the case $d=2$:
we show that every curve $X$ in $\pp^d$ hyperbolic
with respect to some $d-2$-dimensional real linear subspace $V \subset \pp^d$ admits
a definite Hermitian and even real symmetric \vr{} Livsic-type determinantal representation.
Furthermore, when $X$ is irreducible,
the set of all $V' \in \Gr(\ell - 1,d)(\R)$ such that $X$ is hyperbolic with respect to $V'$
is given by a linear matrix inequality (in the coordinates of the Pl\" ucker embedding).

Our terminology is quite standard.
All our varieties are over the field $\C$ of complex numbers,
are reduced unless explicitly stated otherwise, and we identify the variety with the set of its (closed) points
over $\C$.
We say that $X \subset \pp^d$ is a real subvariety if $X$ is defined over the reals
(i.e., by homogeneous polynomial equations with real coefficients);
we then denote by $X(\R)$ the set of points of $X$ that are rational over $\R$
(i.e., have real coordinates).
When we consider the dimension or the codimension of $X$ we assume that $X$ has pure dimension
(i.e., all the irreducible components of $X$ have the same dimension)
unless the converse is explicitly specified.
We denote by $\Gr(m,d)$ the Grassmanian of $m$-dimensional linear subspaces in the $d$-dimensional projective space $\pp^d$.

We will assume that $\C^{d + 1}$ is equipped with the standard scalar product.
For $V$ a subspace in $\C^{d   1}$ we will write $V^{\perp}$ for the orthogonal complement of $V$;
note that if a subspace is real then so is its orthogonal complement.
For most of our purposes $V^{\perp}$ could have been replaced by any complementary subspace, but the use
of the orthogonal complement will streamline some proofs and simplify notations.
We will also use the standard scalar product
to identify $\C^{d   1}$ with its dual,
a fact that we will use later both implicitly and explicitly.

\section{Livsic-type Determinantal Representations} \label{sec:det_rep}

In his work M. S. Livsic and his collaborators considered plane algebraic curves obtained from matrices $\gamma_{01}, \gamma_{02}, \gamma_{12} \in M_n(\C)$ by:
\begin{equation*}
\det \left( \mu_2 \gamma_{01} - \mu_1 \gamma_{02} + \mu_0 \gamma_{12} \right).
\end{equation*}
Now consider the tensor in $\wedge^2 \C^3 \otimes M_n(\C)$ given by $\gamma = \gamma_{01} (e_0 \wedge e_1) + \gamma_{02} (e_0 \wedge e_2) + \gamma_{12} (e_1 \wedge e_2)$, where $e_0$, $e_1$ and $e_2$ form a basis of $\C^3$. For every point $\mu = \mu_0 e_0 + \mu_1 e_1 + \mu_2 e_2  \in \C^3$ one has that:
\begin{equation*}
\gamma \wedge \mu = \left( \mu_2 \gamma_{01} - \mu_1 \gamma_{02} + \mu_0 \gamma_{12} \right) e_0 \wedge e_1 \wedge e_2.
\end{equation*}
Fixing an orientation on $\C^3$, we can identify $\gamma \wedge \mu$ with a matrix in $M_n(\C)$. Note that the determinant of $\gamma \wedge \mu$ is zero if and only if there exists a vector $0 \neq v \in \C^n$, such that $(\gamma \wedge \mu) v = 0$. Furthermore it is invariant under the action of $\C^{\times}$ on $\C^3$ and hence we can identify the curve with the following set of points:
\begin{equation*}
D(\gamma) = \left\{ \mu \in \pp^2 \mid \exists \, v \in \C^n \setminus {0} \,, (\gamma \wedge \mu)v = 0. \right\}. 
\end{equation*}

We will say that a projective plane curve, $X$, admits a Livsic-type determinantal representation if there exists $\gamma \in \wedge^2 \C^3 \otimes M_n(\C)$, such that $X = D(\gamma)$. It has been shown by the second author that every projective plane curve admits a Livsic-type determinantal representation (cf. \cite{Vin89, Vin93, LKMV}).

Each element $\gamma \in \wedge^{k+1} \C^{d+1} \otimes M_n(\C)$ can be thought of as a linear map $\gamma \colon \C^n \to \wedge^{k+1} \C^{d+1} \otimes \C^n$. Fix $e_0,\ldots,e_d$, a basis of $\C^{d+1}$. For $I \subset \{0,\ldots,d\}$ we will write $e_I = e_{i_1} \wedge \ldots \wedge e_{i_r}$, where $I = \{i_1,\ldots,i_r\}$ and $i_1 < i_2 < \ldots < i_r$. Then:
\begin{equation*}
\gamma = \sum_{I \subset \{0,\ldots,d\}, |I| = k + 1} \gamma_I e_I.
\end{equation*}
Thus for $u \in \C^n$ we get $\gamma u = \sum_{I \subset \{0,\ldots,d\}, |I| = k + 1} \gamma_I u \otimes e_I$. Now write $\mu = \sum_{j=0}^d \mu_j e_j$ and for every $J \subset \{0,\ldots,d\}$, $|J| = k+2$ set:
\begin{equation*} 
(\gamma \wedge \mu)_J = \sum_{j \in J} (-1)^{\sigma(J,j)} \mu_j \gamma_{J \setminus \{j\}}.
\end{equation*}
Here $(-1)^{\sigma(J,j)}$ is the sign of the permutation required to obtain the form described above, i.e., $\sigma(J,j) = |\{j^{\prime} \in J \mid j^{\prime} > j \}|$. Conclude that: \begin{equation} \label{eq:gamma_mu_coord}
\gamma \wedge \mu = \sum_{J \subset \{0,\ldots,d\}, |J| = k + 2} (\gamma \wedge \mu)_J e_J.
\end{equation}
Next take $V \subset \pp^d$ a plane of dimension $d-k-1$ spanned by $v_0,\ldots,v_{d-k-1}$. Clearly $\gamma \wedge v_0 \wedge \ldots \wedge v_{d-k-1} \in \wedge^{d+1} \C^{d+1} \otimes M_N(\C)$. We fix an orientation and identify the later space with $\C$ and thus $\gamma \wedge v_0 \wedge \ldots \wedge v_{d-k-1}$ with a matrix. With respect to the fixed basis we have that:
\[
v_0 \wedge \ldots \wedge v_{d-k-1} = \sum_{J \subset \{0,\ldots,d\}, |J| = d-k} p(V)_J e_J.
\]
Here $p(V)_J$ are the coordinates of the vector $v_0 \wedge \ldots \wedge v_{d-k-1}$ with respect to our basis. Hence, using our identification, we can write:
\begin{equation} \label{eq:gamma_v_coord}
\gamma \wedge v_0 \wedge \ldots \wedge v_{d-k-1} = \sum_{I \subset \{0,\ldots,d\}, |I| = k + 1} (-1)^{\sigma(I)} p(V)_{I^c} \gamma_I.
\end{equation}
Here $I^c = \{0,\ldots,d\} \setminus I$ and $(-1)^{\sigma(I)} e_0 \wedge \ldots \wedge e_d = e_I \wedge e_J$. 

Now we can generalize the definition for curves.
\begin{definition}
Given a tensor $\gamma \in \wedge^{k+1} \C^{d+1} \otimes M_n(\C)$, we define the following set:
\begin{equation*}
D(\gamma) = \left\{ \mu \in \pp^d \mid \exists \, v \in \C^n \setminus {0} \,, (\gamma \wedge \mu)v = 0. \right\}. 
\end{equation*}
Here we consider $\gamma \wedge \mu$ as a mapping from $\C^n$ to $\wedge^{k+2} \C^{d+1} \otimes \C^n$. We will say that $\gamma$ is non-degenerate if there exist $v_0,\ldots,v_{d-k-1} \in \C^{d+1}$ linearly independent, such that $\gamma \wedge v_0 \wedge \ldots \wedge v_{d-k-1}$ is invertible, considered as a matrix in $M_n(\C)$.
\end{definition}

Note that non-degeneracy depends only on the $d-k-1$-plane in $\pp^d$, spanned by the vectors $v_0,\ldots,v_{d-k-1}$. Let $V \subset \pp^d$ be this plane, then we denote $\gamma(V) = \gamma \wedge v_0 \wedge \ldots \wedge v_{d-k-1}$. Whenever necessary we will identify $\gamma(V)$ with a matrix in $M_n(\C)$ via an orientation as in \eqref{eq:gamma_v_coord}.

Using \eqref{eq:gamma_mu_coord} we have:
\[
D(\gamma) = \left\{ \mu \in \pp^d \mid \cap_{J \subset \{0,\ldots,d\}, |J| = k + 2} \ker (\gamma \wedge \mu)_J \neq \{0\} \right\}.
\]

\begin{rem}
Note that $D(\gamma)$ is cut out by the ideal generated by the maximal minors of $\gamma \wedge \mu$, considered as a matrix of linear forms in the entries of $\mu$. Alternatively, one can consider it as generated by polynomials of the following form:
\[
\det \left(\sum_{J \subset \{0,\ldots,d\}, |J| = k + 2} m_J (\gamma \wedge \mu)_J\right).\]
Here $m_J \in M_n(\C)$ are arbitrary matrices (cf. \cite[Prop.\ 8.2.1]{LKMV} for the case when $k=1$, the proof of the general case is identical). However, $D(\gamma)$ with this closed subscheme structure will generally be non-reduced and might even have embedded components. We can thus conclude that $D(\gamma)$ is closed subset of $\pp^d$.
\end{rem}

\begin{lem} \label{lem:inclusion}
Fix some tensor $\gamma \in \wedge^{k+1} \C^{d+1} \otimes M_n(\C)$ and a $d-k-1$-plane $V \subset \pp^d$. Then the intersection of $V$ and $D(\gamma)$ is non-empty implies that:
\begin{equation*}
\det \gamma(V) = 0.
\end{equation*} 
\end{lem}
\begin{proof}
Every point on $V$ is of the form $t_0 v_0 + \ldots t_{d-k-1} v_{d-k-1}$, for some basis of $V$. If such a point is on $D(\gamma)$, there exists some non-zero $u \in \C^n$, such that:
\begin{equation*}
\sum_{j=0}^{d-k-1} t_j (\gamma \wedge v_j) u = 0.
\end{equation*}
Now since some $t_j \neq 0$, taking the exterior product with $v_0 \wedge \ldots \wedge \widehat{v}_j \wedge \ldots \wedge v_l$ we get that:
\begin{equation*}
\left(\gamma \wedge v_0 \wedge \ldots \wedge v_{d-k-1} \right) u = 0.
\end{equation*}
\end{proof}

We identify $\Gr(d-k-1,d)$ with its image in $\pp(\wedge^{d-k} \C^{d+1})$ via the Pl\"{u}cker embedding. Recall that the Pl\"{u}cker embedding is the map sending a subspace $V \subset \C^{d+1}$ of dimension $\ell$ to the line $\wedge^\ell V \subset \wedge^{\ell} \C^{d+1}$. Thus we get an embedding of the Grassmannian into $\pp(\wedge^{\ell} \C^{d+1})$. We denote by $v_0 \wedge \ldots \wedge v_{d-k-1}$ the Pl\"{u}cker coordinates of a $d-k-1$-plane $V$ in $\pp^d$. In this setting $\gamma(V)$ defines a matrix of linear forms on the Grassmannian. Note that the $p(V)_J$ in \eqref{eq:gamma_v_coord} are precisely the Pl\"{u}cker coordinates with respect to the basis $e_0,\ldots,e_d$.

\begin{cor}
For a non-degenerate $\gamma$, we have that $\dim D(\gamma) \leq k$. Therefore, for a generic choice of $d-k-1$-plane $V$ we have that $\gamma(V)$ is invertible.
\end{cor}
\begin{proof}
Note that $\det \gamma(V)$ is a section of a line bundle on $\Gr(d-k-1,d)$. Since $\gamma$ is non-degenerate, this section does not vanish identically. Conclude that the zeroes are a hypersurface.
\end{proof}

Let $S = \C[x_0,\ldots,x_d]$ with the natural grading, then $\gamma \wedge \mu$, considered as a matrix of linear forms in the entries of $\mu$, is a map between the graded modules: 
\begin{equation*}
\gamma \wedge \mu \colon S(-1)^n \to S^{n {d+1 \choose k+2}}.
\end{equation*} 

\begin{prop} \label{prop:deg_locus}
The set $D(\gamma)$ is the degeneration locus of a vector bundle map on $\pp^d$. This, in particular, is another way to see that $D(\gamma)$ is closed.
\end{prop}
\begin{proof}
Just apply module to sheaf correspondence for $\Proj$ to the above map, to get:
\begin{equation*}
\gamma \wedge \mu \colon \cO(-1)^n \to  \cO^{n {d+1 \choose k+1}}.
\end{equation*}
The points that belong to $D(\gamma)$ are precisely those points, where the map is not injective on the stalk. Thus $D(\gamma)$ is the degeneration locus of this map.
\end{proof}
Note that the definition is independent of the choice of the coordinates since given any $g \in \GL_{d+1}(\C)$ we have that $\mu \in D(\gamma)$ if and only if $g \mu \in g D(\gamma)$, since the map defined by $\gamma$ changes by a multiplication by an invertible scalar matrix on the left.

\begin{rem}
Following the Beilinson-Gelfand-Gelfand construction one can identify $\wedge^{i-j} \C^{d+1} \cong \Hom(\Omega^i(i),\Omega^j(j))$, for $0 \leq j \leq i \leq 0$. We think of $\Omega^i(i)$ as embedded in $\wedge^i \C^{d+1} \otimes \cO$, where $\cO$ is the sheaf of regular functions on $\pp^d$. Hence in particular $\gamma \in \wedge^{k+1} \C^{d+1} \otimes M_n(\C)$ defines uniquely a map from $\Omega^d(d)^n \cong \cO(-1)^n$ to $\Omega^{d-k-1}(d-k-1)^n$. This map however is not the same map as defined above unless $k = d-1$. There is however a way to change the signs in $\gamma$ to obtain one from the other.
\end{rem}

In general the set $D(\gamma)$ will be empty, unless $d = k + 1$. In order to emphasize a special case when the upper bound is achieved, we make the following definition:

\begin{definition}
A non-degenerate tensor $\gamma \in \wedge^{k+1} \C^{d+1} \otimes M_n(\C)$ will be called \reas{} if $\dim D(\gamma) = k$.
\end{definition}

Recall that an irreducible subvariety $X$ of dimension $k$ of $\pp^d$ defines a class in the $k$-th Chow group of $\pp^d$. It is well known that the $k$-th Chow group of $\pp^d$ is isomorphic to $\Z$ and is generated by the class of a $k$-plane. Therefore $[X] = n [L_k]$ and we call $n$ the degree of $X$. For a pure-dimensional reducible variety, we represent it as a formal sum of its components, therefore its degree is the sum of the degrees of its components. 

It is useful to keep track of the dimension of the kernel of the map $\gamma \wedge \mu$, hence we make the following definition:

\begin{definition}
We define the cycle associated to a non-degenerate $\gamma \in \wedge^{k+1} \C^{d+1} \otimes M_n(\C)$ in $Z_*(\pp^d)$ by:
\begin{equation*}
Z(\gamma) = \sum_{j=1}^r n_j [D_j].
\end{equation*}
Here we denote by $D_j$ the irreducible components of $D(\gamma)$. The numbers $n_j$ are obtained by taking the exact sequence:
\begin{equation*}
0 \to \cO(-1)^n \to  \cO^{n {d+1 \choose k+1}} \to \cC \to 0
\end{equation*}
and pulling it back to $D_j$. Since $D_j$ is in the degeneracy locus, we get the exact sequence:
\begin{equation*}
0 \to \cK \to \cO_{D_j}(-1)^n \to \cO_{D_j}^{n {d+1 \choose k+1}} \to \cC_{D_j} \to 0.
\end{equation*}
We call $\cK$ the kernel sheaf associated to the tensor $\gamma$. The kernel sheaf is a coherent sheaf on $D_j$ and we take $n_j$ to be the dimension of the generic fiber of $\cK$. We define the degree of $\gamma$ to be:
\begin{equation*}
\deg(\gamma) = \int_{\pp^d} [Z(\gamma)][L_{d-k}].
\end{equation*}
Here $[L_{d-k}]$ is the rational equivalence class of the $d-k$-plane in $\pp^d$.
\end{definition}

\begin{rem}
Let $D(\gamma) = D_1 \cup \ldots \cup D_r \cup D^{\prime}$, where $\dim D_j = k$ for each $j$ and they are irreducible and $\dim D^{\prime} < k$. Then a generic $d-k$-dimensional plane $U$ intersects each $D_j$ at $\deg D_j$ distinct points and does not intersect $D^{\prime}$. Let $U_j \subset D_j$ be the open set on which the dimension of the fiber of the kernel sheaf is $n_j$. Since $D_j \setminus U_j$ is a closed subvariety, its dimension is at most $k-1$, hence using the incidence correspondence described below, it is easy to see that a generic $U$ intersects each $D_j$ at points of $U_j$. Hence $\deg \gamma$ is the sum of the dimensions of the fibers of the kernel sheaf at points of intersection with a generic $d-k$-plane. Furthermore, note that for any $d-k$-plane that intersects each $D_j$ at $\deg D_j$ distinct points, the sum of the dimensions of the kernel sheaf fibers at those points is always greater or equal to $\deg \gamma$, since the dimension of the fibers of a coherent sheaf is upper semi-continuous.
\end{rem}

\begin{definition} \label{def:det_rep}
Let $X \subset \pp^d$ be a subvariety of dimension $k$. We say that $X$ admits a Livsic-type determinantal representation if $X = D(\gamma)$ for some non-degenerate tensor $\gamma \in \wedge^{k+1} \C^{d+1} \otimes M_n(\C)$, for some integer $n$. If for some (and hence for every) basis $e_0,\ldots,e_d$ the matrices $\gamma_I$ are symmetric we will say that $X$ admits a symmetric Livsic type determinantal representation. If for some (and hence for every) real basis $e_0,\ldots,e_d$ for $\C^{d+1}$, we have that $\gamma = \sum_{I \subset \{0,\ldots,d\},|I|=k+1} \gamma_I e_I$ with every $\gamma_I$ Hermitian or real symmetric, we will say that $X$ admits a Hermitian or real symmetric Livsic-type determinantal representation, respectively.
\end{definition}

Recall from \cite{Har95} that for every integer $\ell$ we have the incidence correspondence:
\begin{equation*}
\Sigma = \left\{ (x, V) \mid x \in V \right\} \subset \pp^d \times \Gr(\ell,d).
\end{equation*}
We get a diagram by restricting the projection maps to $\Sigma$:
\begin{equation*}
\xymatrix{\Sigma \ar[r]^{p_1} \ar[d]_{p_2} & \pp^d \\ \Gr(\ell,d) & }.
\end{equation*}
Both $p_1$ and $p_2$ are proper and smooth, hence in particular for every closed $X \subset \pp^d$ we have that $p_2(p_1^{-1}(X))$ is closed in $\Gr(\ell,d)$. The fiber of $p_1$ over a point $\mu \in \pp^d$ is isomorphic to $\Gr(\ell-1,d-1)$. The fiber of $p_2$ over $V \in \Gr(\ell,d)$ is isomorphic to $V$ itself. Recall that the dimension of $\Gr(\ell,d)$ is $g_{\ell} = \ell(d-\ell)$. 

Given an irreducible subvariety $X \subset \pp^d$ of dimension $k$ and degree $n$, we know that a generic $d-k-1$-plane does not intersect $X$. Let $\ell = d-k-1$ and $Y = p_2(p_1^{-1}(X))) \subset \Gr(d-k-1,d)$. Furthermore, since generically a $d-k-1$-plane in $\pp^d$ that intersects $X$ does so at a single point, we get that $p_2$ is birational on an open dense subset of $p_1^{-1}(X)$. Since the map $p_1$ is smooth it is in particular flat and of relative dimension $g_{d-k-1} - k -1$. Hence we get a map:
\begin{equation*}
p_1^* \colon A_k(\pp^d) \to A_{g_{d-k-1}-1}(\Sigma).
\end{equation*}
Since $p_2$ is birational on $Y$, $Y$ is a hypersurface in $\Gr(d-k-1,d)$. Furthermore, since $[X] = n[L]$, where $L$ is a $k$-plane in $\pp^d$, we get that:
\begin{equation*}
[Y] = p_{2*} p_1^*([X]) = n p_{2*} p_1^*(L) = n \sigma_1.
\end{equation*}
Here $\sigma_1$ is the first Chern class of the universal quotient bundle on the Grassmannian (one can say that $\sigma_1$ is dual to the rational equivalence class of the intersection of the Grassmannian with a hyperplane in the ambient space of the Pl\"{u}cker embedding). Furthermore, $\sigma_1$ generates $A^1(\Gr(d-k-1,d)) \cong \operatorname{Pic}(\Gr(d-k-1,d))$ (see \cite[Ch.14.6-7]{Ful98}). Since the Grassmannian is non-singular we know that $A^1(\Gr(d-k-1,d)) \cong A_{g_{d-k-1}-1}(\Gr(d-k-1,d))$. Hence the degree of $Y$ equals the degree of $X$. We summarize this discussion in the following well known lemma (see for example \cite{ChvdW37} and \cite[Prop.\ 2.2]{GKZ08}):

\begin{lem} \label{lem:chow-degree}
The hypersurface $Y \subset \Gr(d-k-1,d)$ corresponding to an irreducible subvariety $X \subset \pp^d$  of dimension $k$ under the incidence correspondence is of the same degree as $X$.
\end{lem}

For $u \in \C^{d+1}$ linearly independent from $V$ we denote $\gamma(V,i,u) = \gamma \wedge v_0 \wedge \ldots v_{i-1} \wedge u \wedge v_{i+1} \ldots \wedge v_{d-k-1}$. The following Lemma is a generalization of \cite[Eq.\ 2.24-25]{BV-ZPT}.
\begin{lem} \label{lem:sol-space}
Let $\gamma \in \wedge^{k+1} \C^{d+1} \otimes M_n(\C)$ be non-degenerate. Let $V \subset \pp^d$ be a $d-k-1$-plane, such that $\gamma(V)$ is invertible. Let $u \in \C^{d+1}$ be linearly independent from $V$. Then the intersection points of $U$, the $d-k$-plane spanned by $V$ and $u$, with $D(\gamma)$ are in one-to-one correspondence with a subset of the joint eigenvalues of the matrices $\gamma(V)^{-1} \gamma(V,u,i)$, for $i= 0,\ldots, d-k-1$. Furthermore, the fibers of the kernel sheaf at these points are contained in the corresponding joint eigenspaces and thus are linearly independent as subspaces of $\C^n$.
\end{lem}
\begin{proof}
By Lemma \ref{lem:inclusion}, $V$ does not intersect $D(\gamma)$. However $U$ intersects $D(\gamma)$ in a finite number of points unless $\dim D(\gamma) < k$. Every point in $U \cap D(\gamma)$ is of the form $u + \sum_{j=0}^{d-k-1} t_j v_j$. According to the definition of $D(\gamma)$, there is a vector $w \in \C^n$, such that:
\begin{equation*}
\gamma \wedge \left( u + \sum_{j=0}^{d-k-1} t_j v_j \right) w = 0.
\end{equation*}
Taking the exterior product with $v_0 \wedge \ldots \widehat{v}_i \ldots \wedge v_{d-k-1}$, for some $0 \leq i \leq d-k-1$, we get:
\begin{equation*}
\left( \gamma(V,u,i) - t_i \gamma(V) \right) w = 0.
\end{equation*}
Hence the stalk of the kernel sheaf at each point in the intersection is a subspace of the joint eigenspace of $\gamma(V)^{-1} \gamma(V,u,i)$. We conclude that for distinct points the stalks are linearly independent as subspaces of $\C^n$.
\end{proof}

\begin{cor} \label{cor:degree-bound}
Assume that $\gamma$ is non-degenerate then $\deg(\gamma) \leq n$.
\end{cor}
\begin{proof}
The degree of $\gamma$ is independent of irreducible components of $D(\gamma)$ that are of dimension less than $k$. Hence we may assume that $D(\gamma)$ is of pure-dimension $k$. Let $V$ be such that the $\gamma(V)$ is invertible. For every generic $d-k$-plane through $V$ we have that each irreducible component, $D_j$, is intersected at $\deg(D_j)$ distinct points. Now the dimension of a generic fiber is $n_j$. Applying Lemma \ref{lem:sol-space} we get that the sum of the spaces is direct. Therefore the dimension of the space is $\sum_j n_j \deg(D_j) = \deg(\gamma)$.  Since this is a subspace of $\C^n$ we get that $\deg(\gamma) \leq n$.
\end{proof}

\begin{rem}
Note that if $\gamma$ is not reasonable, then $\deg(\gamma) = 0$.
\end{rem}

\begin{definition}
Given a tensor $\gamma \in \wedge^{k+1} \C^{d+1} \otimes M_n(\C)$, we say that $\gamma$ is \vr{} if $\deg(\gamma) = n$.
\end{definition}

\begin{prop} \label{prop:pure_dim}
If a tensor $\gamma$ is \vr{}, then $D(\gamma)$ is of pure dimension $k$.
\end{prop}
\begin{proof}
By definition $\deg(\gamma) = \int_X [Z(\gamma)][L_{d-k}] = n$. Now if $D(\gamma) = D_1 \cup \ldots \cup D_r$ is the decomposition into irreducible components then $Z(\gamma) = \sum_{j=1}^r n_j D_j$. It suffices to show that if $D(\gamma)$ has an irreducible component $D_{j_0}$ of dimension less than $k$, then $\deg(\gamma) < n$. Fix a point $\mu_0 \in D_{j_0}$ that is not on any other component of $D(\gamma)$. Every $d-k$-plane through $\mu_0$ will be spanned by $\mu_0$ and some $d-k-1$-dimensional plane $V$. Since generically $\gamma(V)$ is invertible, we know that for a generic $d-k$-plane through $\mu_0$ the kernel spaces on the components of dimension $k$ can not span all of $\C^n$. Since the sum of their dimensions is greater or equal to $\deg \gamma$, we conclude that $\deg \gamma < n$.
\end{proof}

Recall that $\gamma(V)$ is a matrix of linear forms on the Grassmannian $\Gr(d-k-1,d)$. One can consider $\gamma(V)$ as a map of vector bundles $\cO_{\Gr(d-k-1,d)}(-1)^n \to \cO_{\Gr(d-k-1,d)}^n$. Let $W \in Z_*(\Gr(d-k-1,d))$ be the cycle of zeroes of the section $\det \gamma(V)$ of $\cO_{\Gr(d-k-1,d)}(n)$, i.e., $W = \sum_{j=1}^r n_j W_j$, where each $W_j$ is a hypersurface and the $n_j$ are the order of zero of $\det \gamma(V)$ on $W_j$. Let us denote by $|W|$ the support of $W$, namely $W = \cup_{j=1}^r W_j$.

Using \eqref{eq:gamma_v_coord} we can write $\det(\gamma(V))$ is a degree $n$ homogeneous polynomial in the coordinate ring of the Grassmannian in Pl\"{u}cker embedding. We can factor this polynomial into irreducible polynomials and each $W_j$ corresponds to an irreducible polynomial and $n_j$ to the multiplicity it appears with in $\det(\gamma(V))$.

Let us recall the definition of a $1$-dimensional Schubert cycle on the Grassmannian. Fix some complete flag $0 \subset U_1 \subset U_2 \subset \ldots U_{d+1} = \C^{d+1}$. The Schubert cycle $L$ is given by:
\begin{equation} \label{eq:schubert}
L = \left\{ V \in \Gr(\ell,d) \mid U_l \subset V \subset U_{l+2}\right\}.
\end{equation}
The next Lemma is the equivalent of Lemma \ref{lem:sol-space} for degeneracy loci on the Grassmannian.

\begin{lem} \label{lem:hypersurface-components}
Let $W$ and $|W|$ be the degeneracy locus of 
$\xymatrix{\cO(-1)^n \ar[r]^T & \cO^n}$ and its support. Denote by $L \subset \Gr(\ell,d)$ a $1$-dimensional Schubert cycle associated to some flag. Then the kernel spaces of $T$ at each of the points in $L \cap |W|$ are linearly independent and generically the intersection of $L$ and $W$ has $n$ points counting multiplicities.
\end{lem}
\begin{proof}
The class of the $1$-dimensional Schubert cycle $[L]$ generates $A^{g_{\ell}-1}$, hence $\int_{\Gr(\ell,d)}[W][L] = n$, so generically it has $n$ points counting multiplicities. Now one let $v_1,\ldots,v_{\ell+2}$ be the basis of $U_{\ell+2}$ such that the first $\ell$ vectors are a basis for $U_{\ell}$, then if $U_{\ell} \subset V \subset U_{\ell+2}$ then the Pl\"{u}cker coordinates of $V$ in $\Gr(\ell,d)$ are $v_1 \wedge \ldots \wedge v_{\ell} \wedge (t v_{\ell+1} + s v_{\ell+2})$, where $[t:s] \in \pp^1$. We may assume that $T(v_1 \wedge \ldots \wedge v_{\ell+1})$ is invertible. Hence passing to the open subset where $s = 1$, we see that:
\begin{equation*}
\det(T(V)) = 0 \iff \det (t I + T(v_1 \wedge \ldots \wedge v_{\ell+1}^{-1}) T(v_1 \wedge \ldots \wedge v_{\ell} \wedge v_{\ell+2})) = 0.
\end{equation*}
The multiplicity of the intersection is the order of zero of the determinant on $L$. The kernels are clearly eigenspaces of a matrix associated to distinct eigenvalues and hence have zero intersections.
\end{proof}

Now we can give a description of \vr{} tensors both geometrically and algebraically:

\begin{thm}  \label{thm:char1}
Let $\gamma \in \wedge^{k+1} \C^{d+1} \otimes M_n(\C)$ be non-degenerate. Let $Y = p_2(p_1^{-1}(D(\gamma))$, $\alpha = p_{2*}(p_1^*(Z(\gamma)) \in Z_*(\Gr(d-k-1,d))$ and $W$ and $|W|$ the degeneracy locus of $\gamma(V)$ and its support. Then the following conditions are equivalent:
\begin{itemize}
\item[(a)] The tensor $\gamma$ is \vr{};

\item[(b)] The variety $Y$ is a hypersurface and furthermore $\alpha = W$ and $Y = |W|$;
\end{itemize}
\end{thm}

\begin{proof}

$(a) \Rightarrow (b)$ By \cite[Ex.\ 11.18]{Har95} if $X \subset \pp^d$ is irreducible, so is $p_2(p_1^{-1}(X))$. So if $D(\gamma) = D_1 \cup \ldots \cup D_r$ is the decomposition into irreducible components and $Y_j = p_2(p_1^{-1}(D_j))$, then $Y = Y_1 \cup \ldots \cup Y_r$ is the decomposition of $Y$ into irreducible components. By Lemma \ref{lem:inclusion} $Y \subset |W|$ and by Proposition \ref{prop:pure_dim} they are of the same dimension. Hence we can conclude that the irreducible components of $Y$ are a subset of the irreducible components of $|W|$. Now $W$ is the degeneracy locus of a map of vector bundles. Take a line as in Lemma \ref{lem:hypersurface-components}; its intersection with $W$ will yield a set of linearly independent subspaces of $\C^n$. Note that for a point $\mu \in D(\gamma)$, if $(\gamma\wedge \mu) u = 0$, then $\gamma(V) u = 0$ for every $d-k-1$-plane $V$ through $\mu$. Hence $\dim \ker \gamma(V) \geq n_j$, for every $V \in Y_j$. However, $\sum_j n_j \deg Y_j = n$ and thus $|W|$ no other components and furthermore $W = \sum_{j=1}^r n_j Y_j$.

$(b) \Rightarrow (a)$ This is immediate since the degree of $W$ is $n$ and $\deg \gamma = \deg \alpha$.
\end{proof}

The following corollary is immediate from the proof.
\begin{cor} \label{cor:vr_kernel_fibers}
Assume $\gamma$ is \vr{}. Let $\cK$ be the kernel sheaf of $\gamma$ on $D(\gamma)$ and let $\cK^{\prime}$ be the kernel sheaf on $|W|$. Then the fibers of $p_{2*}p_1^{*} \cK$ and $\cK^{\prime}$ agree generically.
\end{cor}

\begin{cor} \label{cor:vr_necessary}
Assume that $\gamma$ is \vr{}. Let $V$ a $d-k-1$ plane that does not intersect $D(\gamma)$ and let $U$ be a $d-k$-plane through $V$ that intersects $D(\gamma)$ transversely. Then for every $u\in U$ linearly independent from $V$ the matrices $A_j = \gamma(V)^{-1} \gamma(V,u,j)$ for $j=0,\ldots,d-k-1$, commute and are semi-simple. 
\end{cor}
\begin{proof}
By Theorem \ref{thm:char1} we have that $\gamma(V)$ is invertible. Furthermore, by Lemma \ref{lem:sol-space} we know that for each $\mu \in U \cap D(\gamma)$, the fiber of the kernel sheaf $E_{\mu}$ is a subset of a joint eigenspace of the $A_j$. Again by Theorem \ref{thm:char1} we know that the $E_{\mu}$ span $\C$. We conclude that the $A_j$ commute and are semi-simple.
\end{proof}

To get a sufficient condition we will consider a non-degenerate tensor $\gamma$ and a $d-k-1$-plane $V$, such that $\gamma(V)$ is invertible. Let us assume that $V$ is spanned by $e_{k+1},\ldots,e_d$ and complete it to a basis of $\C^{d+1}$. Recall that a point $\mu = \sum_{j=0}^d z_j e_j \in D(\gamma)$ if there exists a non-zero vector $w \in \C^n$, such that for every $J \subset \{0,\ldots,d\}$ of cardinality $k+2$ we have:
\[
\sum_{j \in J} (-1)^{\sigma(J,j)} z_j \gamma_{J \setminus \{j\}} w = 0.
\]
Note that $\gamma(V)$ is precisely $\gamma_{I_0}$, where $I_0 = \{0,\ldots,k\}$. Hence we get the following equation for every $\ell = k+1,\ldots,d$:
\[
z_{\ell} w = \sum_{j=0}^k (-1)^{\sigma_j} z_j \gamma_{I_0}^{-1} \gamma_{I_0 \setminus \{j\} \cup \{\ell\}} w.
\]
In fact if $\gamma$ is \vr{} this is another way to obtain the result of Corollary \ref{cor:vr_necessary}. Now let $I \subset \{0,\ldots,d\}$ of cardinality $k+1$, such that $|I \cap I_0| \leq k-1$ and let $p \in I_0 \setminus (I \cap I_0)$. Then we can take $J = I \cup \{p\}$ and get the equation:
\[
\sum_{j \in J \cap I_0} (-1)^{\sigma(J,j)} z_j \gamma_{J \setminus \{j\}} w + \sum_{j \in J \setminus (J \cap I_0)} \sum_{\ell = 0}^k (-1)^{\sigma(J,j) + \sigma_{\ell}} z_{\ell} \gamma_{J \setminus \{j\}} \gamma_{I_0}^{-1} \gamma_{I_0 \setminus \{\ell\} \cup \{j\}} w = 0.
\]

The coefficient of $z_p$ is:
\[
\left( \gamma_I + \sum_{j \in J \setminus (J \cap I_0)} (-1)^{\sigma(J,j) + \sigma_{\ell} + \sigma(J,p)} \gamma_{J \setminus \{j\}} \gamma_{I_0}^{-1} \gamma_{I_0 \setminus \{p\} \cup \{j\}} \right) w.
\]
Note that for every $j$ in the sum above we have that $|(J \setminus \{j\}) \cap I_0| = |I \cap I_0| + 1$. So we can express them as well using the same formula. Furthermore, if $\gamma$ is \vr{}, then the variables $z_0,\ldots,z_k$ are free and for every choice of those variables we have a basis for $\C^n$	formed by the joint eigenvectors of the corresponding pencils. Hence, if we take $z_p$ non-zero and others $0$, we'll get that:
\[
\gamma_I = \sum_{j \in J \setminus (J \cap I_0)} (-1)^{\sigma(J,j) + \sigma_{\ell} + \sigma(J,p)} \gamma_{J \setminus \{j\}} \gamma_{I_0}^{-1} \gamma_{I_0 \setminus \{p\} \cup \{j\}}.
\]
It is not difficult to check using induction and the commutation conditions described in Corollary \ref{cor:vr_necessary} that in fact this formula is independent of the choice of $p$. On the other hand it is immediate that if the commutation conditions hold, the matrices described in Corollary \ref{cor:vr_necessary} are semi-simple and the above equations are satisfied, then $\gamma$ is \vr{}.

\section{Hyperbolicity and the Grassmannian} \label{sec:hyper}

Recall that in the classical case a real hypersurface $X \subset \pp^d$ is called hyperbolic with respect to a real point $a \in \pp^d$ if for every real line $L$ that passes through $a$, we have that $X \cap L \subset X(\R)$. 

We will generalize this definition to the case when $\codim X > 1$ as follows:

\begin{definition}
Let $X \subset \pp^d$ be a real subvariety of codimension $\ell$. We'll say that $X$ is hyperbolic with respect to a real linear $\ell -1$-dimensional subspace $V \subset \pp^d$, if $V \cap X = \emptyset$ and for every $\ell$-dimensional subspace, $U$, that contains $V$ we have that $X \cap U \subset X(\R)$.
\end{definition}

\begin{prop} \label{prop:dividing_general}
Assume $X \subset \pp^d$ is a real subvariety of dimension $k$ and $V$ a real $d-k-1$-plane, that does not intersect $X$. Then $X$ is hyperbolic with respect to $V$ if and only if the projection $f$ from $V$ onto $V^{\perp} \cong \pp^k$ restricted to $X$ has the following property: 

$(\star)$ $f(x) \in \pp^k(\R)$ if and only if $x \in X(\R)$. 

\end{prop}
\begin{proof}
Let $V^{\perp}$ be the real $k$-plane associated to the orthogonal complement of $V$. It is immediate that every $d-k$-plane through $V$ intersects $V^{\perp}$ at a single point. Furthermore, if the $d-k$-plane is real then so is its point of intersection with $V^{\perp}$. Consider now the projection of $X$ onto $V^{\perp}$ from $V$, namely for each point $x \in X$, we consider the $d-k$-plane $U_x$ spanned by $x$ and $V$ and map $x$ to the point of intersection of $U_x$ and $V^{\perp}$.  Clearly, if $x \in X(\R)$, then $f(x) \in \pp^k(\R)$, since $U_x$ is real in that case. Next note that for every point $y \in V^{\perp}$ the fiber over $y$ is precisely the points of intersection of the $d-k$-plane $U_y$ spanned by $V$ and $y$ with $X$, hence the map $f$ has property $(\star)$ if and only if $X$ is hyperbolic with respect to $V$.
\end{proof}
Another way to connect the notion of hyperbolicity introduced here and the classical one is similar to the above construction. 

\begin{prop} \label{prop:projection}
Assume $X \subset \pp^d$ is a real subvariety of dimension $k$, and $V$ a real $d-k-1$-plane, that does not intersect $X$. Take $V_0 \subset V$ of codimension $m$ in $V$ and project $\pp^d$ onto $V_0^{\perp} \cong \pp^{k+m}$ as above. Denote the projection by $\pi_{V_0}$. note that $\pi_{V_0}(V)$ is an $m-1$-plane and $\pi_{V_0}(X)$ is subvariety of codimension $m$. Then $X$ is hyperbolic with respect to $V$ if and only if $\pi_{V_0}(X)$ is hyperbolic with respect to $\pi_{V_0}(V)$ for every $V_0 \subset V$ of codimension $m$.
\end{prop}
\begin{proof}
The proof is the same as above.
\end{proof}

\begin{cor} \label{cor:classic}
Let $X, V \subset \pp^d$ be as in Proposition \ref{prop:projection} and let $V_0$ be of codimension $1$ in $V$. Denote by $\pi$ the projection onto $V_0^{\perp}$, then $\pi(X)$ is a real hypersurface hyperbolic with respect to the point $\pi(V)$.
\end{cor}

Let us from now on write $X$ for a real subvariety of $\pp^d$ of pure dimension $k$. Set $\ell = d - k$ and let $Y$ be as in the previous section the real hypersurface in $\Gr(\ell -1,d)$ that corresponds to $X$ via the incidence correspondence. 

The following proposition is immediate from the definitions:
\begin{prop} \label{prop:hyperbolic-equiv}
The subvariety $X$ is hyperbolic with respect to $V$ if and only if for every $1$-dimensional real Schubert cycle (as defined in \eqref{eq:schubert}), $L$, through $V$ in $\Gr(\ell - 1,d)$ we have that $Y \cap L \subset Y(\R)$. In this case we will say that $Y$ is hyperbolic with respect to $V$.
\end{prop}
\begin{proof}
To see this note that every real $1$-dimensional Schubert cycle through $V$ is defined by a real subspace $V_0 \subset V$ of codimension $1$ and a real $d-k$-plane $U$ containing $V$. The points of intersection of $L$ with $Y$ are precisely the $d-k-1$-planes $V^{\prime}$ such that $V_0 \subset V^{\prime} \subset U$ and $V^{\prime} \cap X \neq \emptyset$. Since $V_0$ does not intersect $X$, then $V^{\prime}$ is spanned by the intersection of $U$ with $X$ and $V_0$. Hence the intersections are all real if and only if all of the $V^{\prime}$ are.
\end{proof}

Consider $\Gr(\ell-1,d) \subset \pp^N$ with Pl\"{u}cker embedding. Denote by $\Gr^{\dagger} \subset \R^{N+1}$, one of the connected components of the cone over $\Gr(\ell-1,d)(\R)$. For every $\ell-2$-plane $V_0 \subset \pp^d$ we can define a subset of $\Gr(\ell-1,d)$:
\[
P_{V_0} = \left\{ V^{\prime} \in \Gr(\ell-1,d)(\R) \mid V_0 \subset V^{\prime} \right\}.
\]
Note that $P_{V_0} \cong \pp^{d - \ell -2}(\R)$ and each point $V^{\prime} \in P_{V_0}$ can be identified uniquely with a point on the projection from $V_0$ onto $V_0^{\perp}$.

In fact $P_{V_0}$ is a Schubert cycle of the form $\Omega(U_0,\ldots,U_{d-k-1})$, where we fix a basis $v_0,\ldots,v_{d-k-2}$ for $V_0$ and set $U_j = \operatorname{Span}\{v_0,\ldots,v_j\}$, for $j=0,\ldots,d-k-2$ and $U_{d-k-1}= \pp^d$. This means that its cycle class is $(0,1,\ldots,d-k-2,d)$ and the dual cohomology class is $(k+1,\ldots,k+1,0)$.

\begin{definition} \label{def:sliceconvex}
Let $E \subset \Gr(\ell-1,d)(\R)$. If for a point $V \in E$ and every $V_0 \subset V$ of codimension $1$, we have that the piece of the cone over $E \cap P_{V_0}$ in $\Gr^{\dagger}$ is convex, then we will say that $E$ is slice-convex with respect to $V$. If $E$ is slice-convex with respect to every $V \ \in E$, then we will simply say that $E$ is slice convex.
\end{definition}

\begin{rem}
Consider $E \cap P_{V_0}$ as a subset of $\pp^{d-\ell-2}$ and look at the cone over it in $\R^{d-\ell-1}$. This cone is a union of a pointed convex cone and its negative if and only if the condition of Definition \ref{def:sliceconvex} holds.
\end{rem}

Let $X \subset \pp^d$ be a subvariety of codimension $\ell$ and let $Y \subset \Gr(\ell-1,d)$ be its associated hypersurface. Fix an $\ell-2$-plane $V_0 \subset \pp^d$ that does not intersect $X$ and denote by $\pi_{V_0}$ the projection from $V_0$ onto $V^{\perp}$. Note that every point $V \in P_{V_0} \cap Y$ is an $\ell-1$-plane that intersects $X$ and is spanned by $V_0$ and one of the points in the intersection. On the other hand $\pi_{V_0}(V)$ is the point on $V^{\perp}$ corresponding to $V \cap V_0^{\perp}$. Since $V$ intersects $X$ we have that $\pi_{V_0}(V) \in \pi_{V_0}(X)$. The converse is also true by the definition of the projection. Thus we can identify $P_{V_0} \cap Y$ with $\pi_{V_0}(X)$. Note that this discussion ties together Propositions \ref{prop:dividing_general} and \ref{prop:hyperbolic-equiv}.

\begin{lem} \label{lem:sliceconvex_V}
Let $X \subset \pp^d$ be a real variety of dimension $k$ hyperbolic with respect to some real $d-k-1$-plane $V$. Let $Y \subset \Gr(d-k-1,d)$ be the associated hypersurface and let $C(V)$ be the connected component of $V$ in $\Gr(\ell-1,d) \setminus Y$. Then $C(V)$ is slice-convex with respect to $V$ and furthermore every $X$ is hyperbolic with respect to every $V^{\prime} \in C(V) \cap P_{V_0}$ for every $V_0 \subset V$ of codimension $1$.
\end{lem}
\begin{proof}
Let $V_0 \subset V$ be a real subspace of codimension $1$ in $V$. Projection $\pi_0$ from $V_0$ will map $X$ to a hypersurface hyperbolic with respect to the $\pi_0(V)$. For every other real $d-k-1$-plane $V^{\prime}$, such that $V \cap V^{\prime} = V_0$, $X$ is hyperbolic with respect to $V^{\prime}$ if and only if $\pi_0(X)$ is hyperbolic with respect to $\pi_0(V^{\prime})$. By \cite{Ga59} we know that the cone over the hyperbolicity set of $\pi_0(X)$ consists of two convex cones.
\end{proof}

For the proof of the following theorem we will fix a metric $d$ on $\Gr(\ell-1,d)(\R)$ that induces the classical topology on it. There are several ways to do that, for example we can embed $\Gr(\ell-1,d)(\R)$ in $M_{d+1}(\R)$, by sending a space to the orthogonal projection onto it with respect to the standard scalar product on $\R^{d+1}$. Then the distance between two spaces is the norm of the difference of the associated projections. The important feature of the classical topology on the Grassmannian and hence of the metric is that an open neighborhood of $V$, that is spanned by $v_0,\ldots,v_{\ell-1}$, consists of spaces $V^{\prime}$, spanned by $v^{\prime}_0,\ldots,v^{\prime}_{\ell-1}$, such that each $v^{\prime}_j$ is in a neighborhood of the respective $v_j$. We will formulate this more precisely in the following lemma:

\begin{lem} \label{lem:metric}
Let $V, V^{\prime} \in \Gr(\ell-1,d)(\R)$ if $D(V,V^{\prime}) < \epsilon$, then we can choose orthonormal bases $v_0,\ldots,v_{\ell-1}$ and $v_0^{\prime},\ldots,v_{\ell-1}^{\prime}$ for $V$ and $V^{\prime}$, respectively, such that $\|v_j - v_j^{\prime}\| < \sqrt{2} \epsilon$, for every $j = 0, \ldots, \ell -1$. Conversely take $V, V^{\prime} \in \Gr(\ell - 1,d)(\R)$, such that $V \cap V^{\prime}$ has an orthonormal basis $v_0,\ldots,v_r$ ($r = -1$ if the intersection is trivial) and complete it to orthonormal bases $v_0,\ldots,v_{\ell-1}$ and $v_0,\ldots,v_r,v_r^{\prime},\ldots,v_{\ell-1}^{\prime}$ for $V$ and $V^{\prime}$, respectively, then if $\| v_j - v_j^{\prime}\| < \delta$ for $j = r,\ldots,\ell-1$, then $d(V,V^{\prime}) < 2 (\ell - r - 1) \delta$.
\end{lem}
\begin{proof}
For the first part fix orthonormal bases $v_0,\ldots,v_{\ell-1}$ and $v_0^{\prime},\ldots,v_{\ell-1}^{\prime}$ for $V$ and $V^{\prime}$, respectively. Note that every $v^{\prime} \in V^{\prime}$ we can write uniquely as $v^{\prime} = v + u$, where $V \in V$ and $u \in V^{\perp}$. Take a unit vector $v^{\prime} \in V^{\prime}$, then by assumption $\|u\| = \|P v^{\prime} - P^{\prime} v^{\prime}\| < \epsilon$, where $P$ and $P^{\prime}$ are orthogonal projections onto $V$ and $V^{\prime}$,respectively. Now write $v_0^{\prime} = \sum_j \alpha_j v_j + u$ and let us assume that $\alpha_0 > 0$ (otherwise replace $v_0$ with $-v_0$). Then:
\[
\langle v_0^{\prime} - v_0 ,v_0^{\prime} - v_0 \rangle = 2 ( 1 - \alpha_0 ).
\]
On the other hand $\|u\| < \epsilon$. Now using the fact that $v_0$ is normal we get:
\[
1 = \|v_0\|^2 = \sum_j \alpha_j^2 + \|u\|^2 < \alpha_0^2 + \epsilon^2.
\]
Hence we get that $ 1 - \alpha_0^2 < \epsilon^2$. Next note that $1 - \alpha_0^2 \geq 1 - \alpha_0$, since $0 \leq \alpha_0 \leq 1$. Thus $1 - \alpha_0 < \epsilon^2$ and $\|v_0^{\prime} - v_0\| < \sqrt{2} \epsilon$. Similarly for every other index.

For the second part take any unit vector $u$ and write $P u = \sum_{j=0}^{\ell-1} \langle u, v_j \rangle v_j$ and similarly $P^{\prime} u = \sum_{j=0}^{r} \langle u , v_j \rangle v_j + \sum_{j=r+1}^{\ell-1} \langle u, v_j^{\prime} \rangle v_j^{\prime}$. Write $\alpha_j = \langle u,v_j \rangle$ and $\alpha^{\prime} = \langle u, v_j^{\prime} \rangle$. Then we have:
\[
\| Pu - P^{\prime} u\| \leq \sum_{j=r+1}^{\ell-1} \|\alpha_j v_j - \alpha_j^{\prime} v_j^{\prime} \|.
\]
Next note that using Cauchy-Schwartz and the fact that $u$ is a unit vector we get that $|\alpha_j - \alpha_j^{\prime}| < \delta$. Therefore:
\[
\|\alpha_j v_j - \alpha_j^{\prime} v_j^{\prime} \| = \| (\alpha - \alpha^{\prime}) v_j + \alpha^{\prime} (v_j - v_j^{\prime}) \| < \delta + |\alpha^{\prime}| \delta < 2 \delta.
\]
Hence $\| P - P^{\prime}\| < 2 (\ell - r -1) \delta$. 
\end{proof}

For simplicity, if $X$ is hyperbolic with respect to $V$, we will say that $V$ witnesses the hyperbolicity of $X$ or shortly that $V$ is a witness.

\begin{thm} \label{thm:connected_component}
Assume that $X$ is hyperbolic with respect to $V$, then $X$ is hyperbolic with respect to every $V^{\prime} \in C(V)$.
\end{thm}
\begin{proof}
We will prove the claim in two steps, first we'll show that the set of all witnesses is open. Then we will use a metric argument to show that in fact every $V^{\prime} \in C(V)$ is a witness.

For the first argument take a ball with radius $\epsilon > 0$ around $V$ that is contained in $C(V)$ and take a point $V^{\prime} \in C(V)$ such that $d(V,V^{\prime}) < \epsilon/2 \sqrt{2}\ell$. Fix orthonormal bases $v_0,\ldots, v_{\ell-1}$ and $v^{\prime}_0,\ldots,v^{\prime}_{\ell-1}$ for $V$ and $V^{\prime}$, respectively. By Lemma \ref{lem:metric} we know that $\|v_j - v_j\| < \epsilon/ 2 \ell$, for $j=0,\ldots,\ell-1$. Now set $V_0$ the space spanned by $v_1,\ldots,v_{\ell-1}$ and $W_0$ the space spanned by $V_0$ and $v^{\prime}_0$. Applying Lemma \ref{lem:metric} again we get that $d(V,W_0) < \epsilon/\ell$ and clearly $W_0 \in P_{V_0} \cap C(V)$, hence in particular $W_0$ is a witness.  Now we proceed inductively each time replacing a single basis vector and the distance between each two consecutive points will be less than $\epsilon/\ell$. Therefore by the triangle inequality they are all contained in the ball with radius $\epsilon$ around $V$. This shows that the set of witnesses contains the ball with radius $\epsilon/2 \sqrt{2} \ell$ around $V$ and thus it is open. 

Take now any $V^{\prime} \in C(V)$ and since the Grassmannian is path connected , we can connect it with a simple path $p \colon [0,1] \to \Gr(\ell-1,d)(\R)$ to $V$, $p(0) = V$, that is contained in $C(V)$. Let $\epsilon > 0$ be the distance from the path to the associated hypersurface, that is defined since both are compact. Since the path $p$ is continuous from a compact set it is uniformly continuous hence there exists $\delta > 0$, such that if $|t -s| < \delta$ then $d(p(t),p(s)) < \epsilon/2\sqrt{2}\ell$. By the first part $p([0,\delta)$ consists of witnesses now just cover $[0,1]$ by segments of length $\delta$ and apply the first part repeatedly to see that $p(1) = V^{\prime}$ is a witness.
\end{proof}

\begin{cor} \label{cor:sliceconvex}
The set $C(V)$ is slice convex.
\end{cor}
\begin{proof}
Apply Corollary \ref{lem:sliceconvex_V} to each and every $V \in C(V)$.
\end{proof}

There is a connection between hyperbolicity and determinantal representations encoded in the following proposition.

\begin{prop} \label{prop:definite-hyper}
Assume $X$ admits a \vr{} Hermitian Livsic-type determinantal representation, $\gamma \in \wedge^{k+1} \C^{d+1} \otimes M_n(\C)$. Assume, furthermore, that for some real $V$, we have that $\gamma(V)$ is positive definite, then $X$ is hyperbolic with respect to $V$.
\end{prop}
\begin{proof}
Let $U$ be a real $\ell$-plane containing $V$. Fix a basis $v_0,\ldots,v_{\ell}$ for $V$ and add a vector $u$ to complete it to a basis of $U$. Since $\gamma(V)$ is positive definite, it is invertible and therefore by Lemma \ref{lem:inclusion} we know that $V \cap X = \emptyset$. Let $u + \sum_{j=0}^{\ell-1} t_j v_j$ be a point of intersection of $U$ and $X$. By definition we have a $w \in \C^n$, such that:
\[
\left(\gamma \wedge \left( u + \sum_{j=0}^{\ell-1} t_j v_j \right) \right) w = 0.
\]
Recall that by Lemma \ref{lem:sol-space} we have that the $t_j$ are eigenvalues of $\gamma(V)^{-1} \gamma(V,u,i)$. Note that $\gamma(V,u,i)$ is Hermitian, since the representation is Hermitian and thus $\gamma(V,u,i)$ is a linear combination with real coefficients of Hermitian matrices. Since $\gamma(V)$ is positive definite, conclude that $\gamma(V)^{-1} \gamma(V,u,i)$ is also Hermitian and thus all its eigenvalues are real.
\end{proof}

Hyperbolicity of hypersurfaces has been studied extensively since the notion was introduced. A few analogous questions arise in our setting.

Consider the Grassmannian embedded in $\pp^N$ via the Pl\"{u}cker embedding. This embedding is projectively normal and since $\operatorname{Pic}(\Gr(\ell-1,d)) \cong \Z$ we obtain that every hypersurface in the Grassmannian is obtained via an intersection with a hypersurface in $\pp^N$. In a series of works H. Buseman discussed various notions of convexity of a subset of the cone over Grassmannian (in the Pl\"{u}cker embedding), e.g. \cite{Bus61}. In particular he calls extendably convex those sets in the cone that are intersections with convex sets in the ambient space. 

\begin{quest}
Let $X$ be an irreducible real variety in $\pp^d$ with $\codim X = \ell > 1$ and let $Y \subset \Gr(\ell-1,d)$ be the hypersurface associated to $X$ via the incidence correspondence. Is it true that the cone over $C(V)$ intersection $\Gr^{\dagger}$ is an extendably convex set in $\R^{N+1}$? Furthermore is it true that $C(V)$ coincides with the set of all witnesses to the hyperbolicity of $X$?
\end{quest}

In Section \ref{sec:hyper_curve} we will show that in the case of curves the cone over the set of all witnesses intersection $\Gr^{\dagger}$ is extendably convex. However, we do not know yet whether this set coincides with $C(V)$ even in that case.

\begin{quest}[Generalized Lax Conjecture, cf. \cite{Vppf}]
Assume we have a real variety $X \subset \pp^d$ of dimension $k$ hyperbolic with respect to a real $d-k-1$-plane $V$, does there exist a real hyperbolic $X^{\prime} \subset \pp^d$, such that $X \cup X^{\prime}$ admits a Livsic-type Hermitian determinantal representation $\gamma \in \wedge^{k+1} \C^{d+1} \otimes M_n(\C)$, such that $\deg X \cup X^{\prime} = n$ and $\gamma(U)$ is definite for a real $d-k-1$-plane $U$ if and only if $U \in C(V)$?
\end{quest}

In Section \ref{sec:hyper_curve} we will obtain such a (multi)linear matrix inequality representation in the case where $X$ is an irreducible curve, without any auxiliary variety $X^{\prime}$, for the set of all witnesses instead of $C(V)$.

\section{Bezoutians of Meromorphic Functions on a Riemann Surface} \label{sec:bezoutian}

Let $X$ be a compact Riemann surface of genus $g$. Fix a canonical basis for the homology of $X$, $A_1,\ldots,A_g,B_1,\ldots,B_g$ and fix a normalized basis for holomorphic differentials, $\omega_1,\ldots,\omega_g$. Normalization means that $\int_{A_j} \omega_i = \delta_{ij}$. Set $\Omega$, the $B$-period matrix, given by columns of the form $\begin{pmatrix} \int_{B_j} \omega_1 & \cdots & \int_{B_j} \omega_g \end{pmatrix}^{T}$. Then $J(X) = \C^g/\left(\Z^g + \Omega Z^g\right)$ is the Jacobian variety of $X$. Fix a point $p_0 \in X$ and set $\varphi \colon X \to J(X)$ the Abel-Jacobi map, given by:
$$
\varphi(p) = \begin{pmatrix}
\int_{p_0}^p \omega_1, \cdots, \int_{p_0}^p \omega_g
\end{pmatrix}.
$$
Extend $\varphi$ linearly to all divisors on $X$. Thus by writing $\varphi(\cL)$ for a line bundle $\cL$ on $X$, we mean the image of the corresponding divisor.

Fix a line bundle of half-order differentials $\Delta$ on $X$, such that $\varphi(\Delta)= -\kappa$, the Riemann constant. Additionally fix a flat line bundle $\chi$ on $X$, such that $h^0(\chi \otimes \Delta) = 0$. Since $\chi$ is flat, the sections of $\chi$ lift to functions on $\tilde{X}$, the universal cover of $X$, that satisfy for every $T \in \pi_1(X)$ and every $\tilde{p} \in \tilde{X}$:
\[
f( T \tilde{p}) = a_{\chi}(T) f(\tilde{p}).
\]
Here $a_{\chi}$ is the constant factor of automorphy associated to $\chi$. In fact a choice of a trivialization of $\chi$ in a neighborhood of a point $p$ is equivalent to a choice of a lift $\tilde{p} \in \tilde{X}$. We can also lift $\varphi$ to a map from $\tilde{X}$ to $\C^g$. Since every $\tilde{p}$ is represented by a point $p \in X$ and a path $c$ connecting $p_0$ to $p$, then:
\[
\varphi(\tilde{p}) = \begin{pmatrix}
\int_c \tilde{\omega}_1, \cdots, \int_c \tilde{\omega}_g
\end{pmatrix}.
\]
The differential $\tilde{\omega}_j$ is the pullback of $\omega_j$ to $\tilde{X}$ via the coveting map.

Let us write $\theta(z)$ for the theta function associated to the lattice $\Z^g + \Omega Z^g$, where $\Omega$ is the period matrix of $X$, namely:
\[
\theta(z) = \sum_{m \in \Z^g} e^{2 \pi i \langle \Omega m, m \rangle + 2 \pi i \langle z, m \rangle}.
\]
We will also need the theta function with characteristic, so for $a,b \in \R^g$ we define:
\[
\theta{a \brack b}(z) = \sum_{m \in \Z^g} e^{\pi i \langle \Omega( m + a), m+a \rangle} e^{2 \pi i \langle z + b, m + a \rangle}.
\]

Recall from \cite{BV-ZPF} that there exists a Cauchy kernel $K(\chi,p,q)$ a meromorphic map of line bundles on $X\times X$ with only a simple pole along the diagonal with residue $1$,  given by:
\begin{equation}
K(\chi,p,q) = \frac{\theta{a \brack b}(\varphi(q) - \varphi(p))}{\theta{a \brack b}(0) E_{\Delta}(q,p)}.
\end{equation}
Where $\varphi(\chi) = b + \Omega a$ and $E_{\Delta}(\cdot,\cdot)$ is the prime form $X \times X$, with respect to $\Delta$. Pulling back $K(\chi,\cdot,\cdot)$ to $\tilde{X}$, we get a section of the pullback of $\Delta$ satisfying:
\[
\frac{K(\chi,T \tilde{p}, R \tilde{q})}{\sqrt{dt}(T \tilde{p}) \sqrt{ds}(R \tilde{q})} = a_{\chi}(T) \frac{K(\chi,\tilde{p},\tilde{q})}{\sqrt{dt}(\tilde{p}) \sqrt{ds}(\tilde{q})} a_{\chi}(R)^{-1}.
\]
See \cite{AV} for details. Here $t$ and $s$ are local coordinates on $X$ centered at $p$ and $q$, respectively, and $T,R \in \pi_1(X)$. The pullback is holomorphic at $(\tilde{p},\tilde{q})$, as long as $p \neq q$.

Let $f$ and $g$ be two meromorphic functions with simple poles.  We define a meromorphic section of $\Hom(\pi_2^* \chi, \pi_1^* \chi \otimes \pi_1^* \Delta \otimes \pi_2^* \Delta)$ on $X \times X$:
$$
b_{\chi}(f,g)(p,q) = \left( f(p) g(q) - f(q) g(p) \right) K(\chi, p, q).
$$
Assume that $p$ is not a pole of either $f$ or $g$ and fix a local coordinate $t$ centered at $p$. Now if $q$ tends to $p$ we get:
\begin{multline*}
f(p) (g(p) + g^{\prime}(p) t + \ldots) - (f(p) + f^{\prime}(p) t + \ldots) g(p) \to \\ (f(p) g^{\prime}(p) - f^{\prime}(p) g(p)) t + \ldots
\end{multline*}
Since the residue of $K(\chi,\cdot,\cdot)$ along the diagonal is $1$ we get that:
$$
b_{\chi}(f,g)(p,q) \to f(p) g^{\prime}(p) - f^{\prime}(p) g(p).
$$
Note that this is independent of the choice of the lifts of $p$ and $q$, since when $\tilde{q}$ will go to $\tilde{p}$, the factors of automorphy will cancel out in the limit.

Now observe that since $K(\chi,\cdot,\cdot)$ is holomorphic off the diagonal, we get that $b_{\chi}(f,g) = 0$ if and only if $f/g = const.$. Indeed if we fix $p$ that is neither a pole nor a zero for either $f$ or $g$, we get that for every $q$ in an open set in $X$ we have the equality:
$$
\frac{f(p)}{g(p)} = \frac{f(q)}{g(q)}.
$$
Note that $b_{\chi}$ is alternating and linear as a function of $f$ and $g$. Hence we have that:
$$
b_{\chi}(\alpha_1 f + \beta_1 g, \alpha_2 f + \beta_2 g) = (\alpha_1 \beta_2 - \alpha_2 \beta_1) b_{\chi}(f,g).
$$
Given a set of points $S = \{p_1, \ldots,p_m\} \subset X$, we define an effective reduced divisor $D = \sum_{j=1}^m p_j$. Recall that $\cL(D)$ is the space of all meromorphic functions $f$ on $X$ satisfying $(f) + D \geq 0$. In other words that is to say that $f$ has at most simple poles on $S$ and is holomorphic on $X \setminus S$. Furthermore, for every point $p_j$ we fix a lift $\tilde{p}_j$ to $\tilde{X}$ and fix a local coordinate $t_j$ centered at $p_j$ and a corresponding local holomorphic frame $\sqrt{dt_j}$ of $\Delta$. Then we define for every point $p \in X \setminus S$ two sections of $\left(\chi \otimes \Delta\right)^{\oplus m}$:
\begin{equation*}
\bucl(p) = \begin{pmatrix}
\frac{K(\chi,p,\tilde{p}_1)}{\sqrt{dt_1}(\tilde{p}_1)} & \cdots & \frac{K(\chi, p, \tilde{p}_m)}{\sqrt{dt_m}(\tilde{p}_m)}
\end{pmatrix} \mbox{ and } \buc(p) = \begin{pmatrix}
K(\chi,\tilde{p}_1,p)/\sqrt{dt_1}(\tilde{p}_1) \\ \vdots \\ K(\chi,\tilde{p}_m,p)/\sqrt{dt_m}(\tilde{p}_m)
\end{pmatrix}.
\end{equation*}
Note that changing the lift $\tilde{p}_j$ will result in the multiplication of $\buc(p)$ and $\bucl(p)$ by the a diagonal matrix of the constant factors of automorphy. Changing the coordinates will result in multiplication by a diagonal matrix of transition function for $\Delta$ at $p$.

\begin{prop} \label{prop:bezoutians_formula}

Set $D = \sum_{j=1}^m p_j$ be an effective reduced divisor and let $f,g \in \cL(D)$. Then there exists a matrix $B_{\chi,D}(f,g) \in M_m(\C)$, such that for $p \neq q$:
\begin{equation} \label{eq:bezoutian_fundamental}
b_{\chi}(f,g)(p,q) = \bucl(p) B_{\chi,D}(f,g) \buc(q) = \sum_{i,j=1}^m b_{ij} \frac{K(\chi,p,\tilde{p}_i)K(\chi,\tilde{p}_j,q)}{\sqrt{dt_j}(\tilde{p}_j)\sqrt{dt_i}(\tilde{p}_i)}.
\end{equation}
Whereas when $p = q$ is not a pole of either $f$ or $g$, we fix a coordinate $t$ centered at $p$ and get the limit version:
\begin{equation} \label{eq:bezoutian_fundamental_lim}
\bucl(p) B_{\chi,D}(f,g) \buc(p) = f(p) g^{\prime}(p) - f^{\prime}(p) g(p).
\end{equation}
Here for every $j$, the $t_j$ are local coordinates centered at $p_j$. The equality is to be understood literally if neither $p$ or $q$ are poles of $f$ or $g$ and as a limit in case at least one of them is a pole.
\end{prop}
\begin{proof}
Let us fix a point $q$ not in $D$. Let $t$ be a local coordinate centered at $q$. The map $\rho \colon X \to X \times X$, defined by $p \mapsto (p,q)$, satisfies $\pi_1 \circ \rho = 1_X$ and $\pi_2 \circ \rho(p) = q$. Hence $\rho^* \pi_1^* \cF = \cF$ and $\rho^* \pi_2^* \cF = \cF_q$, for every sheaf $\cF$ on $X$. Hence if we divide out by $1/\sqrt{dt}(q)$, we'll get that both sides of \eqref{eq:bezoutian_fundamental} are meromorphic sections of $\chi \otimes \Delta$. Since this line bundle admits no holomorphic sections, except for $0$, it suffices to show that both sections have the same poles and identical principal parts at these poles. Clearly the poles of the sections thus obtained are precisely the poles of $f$ and $g$ on the left-hand side and $D$ on the right hand side. If $p_i$ is not a pole of either $f$ or $g$, we'll set $b_{ij} = b_{ji}  = 0$, for every $j$. Therefore we may assume that $D$ consists precisely of poles of either $f$ or $g$.

Write the Laurent expansion of $f$ and $g$, with respect to $t_j$:
$$
f(t_j) = \frac{a_j}{t_j} + b_j + \ldots
$$
$$
g(t_j) = \frac{c_j}{t_j} + d_j + \ldots
$$
Now we set:
$$
b_{ij} = \left\{\begin{array}{cc} (a_i c_j - a_j c_i)\frac{K(\chi,\tilde{p}_i,\tilde{p}_j)}{\sqrt{dt_i}(\tilde{p}_i)\sqrt{dt_j}(\tilde{p}_j)}, & i \neq j \\a_i d_i - b_i c_i , & i=j \end{array} \right.
$$

For a fixed $i_0$, we have that the left hand side of \eqref{eq:bezoutian_fundamental} (multiplied by $1/\sqrt{dt}(q)$) has a simple pole with residue $(a_{i_0} g(q) - c_{i_0} f(q))K(\chi,p_{i_0},q)$. The right hand side has also a simple pole with residue $\sum_{j=1}^m b_{i_0 j} K(\chi,p_j,q)$. Both expressions can be considered as maps from $\chi$ to $\Delta$ or in other words, sections of $\chi^{\vee} \otimes \Delta$, the Serre dual of $\chi \otimes \Delta$. By Riemann-Roch we get that $h^0(\chi^{\vee} \otimes \Delta) = 0,$ as well. Hence we can apply similar considerations to those sections. We note that the poles are precisely $S$ and computing the residues we obtain the equality with $b_{i_0 j}$ defined above.

To get \eqref{eq:bezoutian_fundamental_lim} we fix a lift $\tilde{p}$ of $p$ and pass to the limit. Next we note that since changing $\tilde{p}$ to $T \tilde{p}$ will result in cancellation, the equality is independent of the choice of the lift.
\end{proof}

\begin{rem}
In particular note that due to cancellation of the constant factors of automorphy appearing when we change the choice of $\tilde{p}_j$, we get that the formula is independent of the choice of those lifts.
\end{rem}

This leads us to the following definition:

\begin{definition}
We define the Bezoutian of the functions $f$ and $g$ with respect to the divisor $D$ as the matrix $B_{\chi,D}(f,g)$.
\end{definition}
One can see immediately from the proof that is $D \leq D^{\prime}$ are two effective reduced divisors on $X$ and $f,g \in \cL(D)$, then $B_{\chi,D}(f,g)$ is a submatrix of $B_{\chi,D^{\prime}}$ and furthermore $B_{\chi,D^{\prime}}(f,g)$ is obtained by padding $B_{\chi,D}(f,g)$ with zeroes to the required size.

There are a few choices made in the construction of the Bezoutian matrix. Changing the lifts of the $p_j$ will result in the conjugation of $B_{\chi,D}(f,g)$ by diagonal unitary matrices of the constant factors of automorphy of $\chi$. Similarly changes in coordinates result in conjugation by the respective matrices of transition functions.

The following corollary is immediate from the definition of the Bezoutian:
\begin{cor}
Let $D = \sum_{j=1}^m p_j$ be an effective reduced divisor on $X$. Then the Bezoutian defines a linear map
\[
B_{\chi,D} \colon \wedge^2 \cL(D) \to M_m(\C).
\]
\end{cor}

\begin{prop} \label{prop:bezotian_symmetric}
If $\chi \otimes \chi \cong \cO$, i.e., $\varphi(\chi)$ is a half-period, then $B_{\chi,D}(f,g)$ is a complex symmetric matrix.
\end{prop}
\begin{proof}
If $\varphi(\chi)$ is a half-period and it is off the theta divisor, it must be an even characteristic. Hence the resulting theta function is even. Since the prime form is anti-symmetric, we get that $K(\chi,p,q) = -K(\chi,q,p)$ in this case. Therefore $b_{\chi}(f,g)(p,q) = b_{\chi}(f,g)(q,p)$ and thus the resulting Bezoutian is symmetric.
\end{proof}

The following proposition shows that the pullbacks of $\buc$ and $\bucl$ to $\tilde{X}$ as vector valued functions have certain duality and independence properties.
\begin{prop} \label{prop:u_independence}
Let $D= \sum_{j=1}^m p_j$ be an effective and reduced divisor on $X$. Let $g \in \cL(D)$ of degree $r$ and take an unramified fiber $g^{-1}(z) = \{q_1,\ldots,q_r\}$, for some $z \in \C$. For each $j=1,\ldots,r$, fix a lift $\tilde{q}_j \in \tilde{X}$. Then the vectors $\buc(\tilde{q}_j)$ are linearly independent and the same is true for $\bucl(\tilde{q}_j)$.

Set $W = \Span \{ \buc(\tilde{q}_1),\ldots, \buc(\tilde{q}_r)\}$ and $W_{\ell} = \Span \{ \bucl(\tilde{q}_1),\ldots, \bucl(\tilde{q}_r)\}$. For every $f \in \cL(D)$ that does not vanish at $q_1,\ldots,q_r$, the matrix $B_{\chi,D}(f,g)$ defines a non-degenerate pairing between the subspaces $W$ and $W_{\ell}$. For every $v \in W$ and $v_{\ell} \in W_{\ell}$ set $[v,v_{\ell}] = v_{\ell} B_{\chi,D}(f,g) v$. Then the $\buc(\tilde{q}_j)$ and $\bucl(\tilde{q}_j)$ are dual with respect to that pairing.
\end{prop}
\begin{proof}
Without loss of generality we may assume that $z = 0$, otherwise replace $g$ with $g-z$ and note that both $\buc$ and $\bucl$ are independent of $g$. Take some $f \in \cL(D)$ such that $f$ does not vanish on $q_1,\ldots,q_r$. Then for every $i \neq j$ we have that:
$$
\bucl(\tilde{q}_i) B_{\chi,D}(f,g) \buc(\tilde{q}_j) = 0.
$$
On the other hand we have that since those are simple zeroes of $g$, we get:
$$
\bucl(\tilde{q}_i) B_{\chi,D} \buc(\tilde{q}_i) = - f(q_i) g^{\prime}(q_i) \neq 0.
$$
Assume that there exist constants $\alpha_1,\ldots,\alpha_r \in \C$, such that the linear combination $\sum_{j=1}^r \alpha_j \buc(\tilde{q}_j) = 0$. Then premultiplying by $\buc(\tilde{q}_i)$, we get that $\alpha_i = 0$. Conclude that the vectors are linearly independent. Similarly for $\bucl$.
\end{proof}

Note that the result is independent of the choices in the construction of $\buc$ and $\bucl$, since the difference is multiplication by an invertible matrix. 

\begin{cor} \label{cor:u_independence_divisor}
Let $D = \sum_{j=1}^m p_j$ be an effective reduced divisor on $X$ and assume that $D$ is precisely the divisor of poles of a meromorphic function $g$. Let $z \in \C$, be such that $g$ is unramified over $z$ and set $g^{-1}(z) = \{q_1,\ldots,q_m\}$. For every $j=1,\ldots,m$, fix a lift $\tilde{q}_j \in \tilde{X}$. Then $\buc(\tilde{q}_j)$ and $\bucl(\tilde{q}_j)$ span $\C^m$ and are dual bases with respect to $B_{\chi,D}(1,g)$.
\end{cor}

\begin{cor} \label{cor:common_zero}
Let $D$ be as in Corollary \ref{cor:u_independence_divisor} and $f,g \in \cL(D)$. Then if $f$ and $g$ have a common zero at $p$, then $B_{\chi,D}(f,g)\buc(\tilde{p}) = 0 $. Independently of the choice of the lift $\tilde{p}$. 
\end{cor}
\begin{proof}
Note that by choosing an appropriate constant $c$, the set $S$ is the divisor of poles of the function $f + c g$. By assumption, for every $q \in X$, we have that:
\[
b_{\chi}(f,g)(p,q) = 0.
\]
Therefore, by Proposition \ref{prop:bezoutians_formula} we conclude that:
\[
\bucl(\tilde{q}) B_{\chi,D}(f,g) \buc(\tilde{p}) = 0.
\]
Now by the assumption on $D$ and Corollary \ref{cor:u_independence_divisor} we conclude that there exist $q_1,\ldots,q_m$, such that $\bucl(\tilde{q}_j)$ are a basis for $\C^m$ dual to $\buc(\tilde{q}_j)$, hence $B_{\chi,D}(f,g) \buc(\tilde{p}) = 0$.

Note that if we replace $\tilde{p}$ by $T \tilde{p}$ for some $T \in \pi_1(X)$, then $\buc(T \tilde{p}) = a_{\chi}(T) \buc(p)$ and hence $B_{\chi,D}(f,g)\buc(\tilde{T p}) = 0$ as well.
\end{proof}
\begin{cor} \label{cor:det_vanish}
Let $D$ be as in Corollary \ref{cor:u_independence_divisor} and $f,g \in \cL(D)$. Then for every $p \in X \setminus S$, we have that:
$$
\left( g(p) B_{\chi,D}(1,f) - f(p) B_{\chi,D}(1,g) + B_{\chi,D}(f,g) \right) \buc(\tilde{p}) = 0.
$$
This equality is independent of the choice of $\tilde{p}$.
\end{cor}
\begin{proof}
Note that by the anti-symmetry of the Bezoutian we have that:
$$
B_{\chi,D}( f - f(p), g - g(p) ) = g(p) B_{\chi,D}(1,f) - f(p) B_{\chi,D}(1,g) + B_{\chi,D}(f,g).
$$
Now just apply Corollary \ref{cor:common_zero}.
\end{proof}

\section{Real Riemann Surfaces and Dividing Functions} \label{sec:div_func}

In this section we'll keep the notations from the previous section and assume that $X$ is equipped with an anti-holomorphic involution $\tau$. Set $X(\R)$ the set of fixed points of $\tau$. Recall that $X$ is called dividing if $X \setminus X(\R)$ has two connected components.

\begin{definition}
We say that a meromorphic function $f$ is dividing, if $f(p) \in \pp^1(\R)$ if and only if $p \in X(\R)$.
\end{definition}. 
Note that if $X$ admits such a function then clearly $X$ is dividing and the two components of $X \setminus X(\R)$ are given by $X_{+} = \{ p \in X \mid \operatorname{Im} f(p) > 0\}$ and $X_{-} = \{ p \in X \mid \operatorname{Im} f(p) < 0\}$. The converse is also true, see \cite{Ahl47} and \cite[Sec.\ 4]{Ahl50}.
Let us call the orientation induced on $X(\R)$ from $X_{+}$, positive. If $p \in X(\R)$ and $t$ is a real coordinate centered at $p$ then the Laurent expansion of $f$ with respect to $t$ will have real coefficients. Furthermore, if we consider the function $f \circ t^{-1}$ as a meromorphic function on a disc, then it takes points with positive (resp. negative) imaginary parts to points with positive (resp. negative) imaginary parts as well.

The following proposition is in fact contained in \cite{Ahl47} and \cite{Ahl50}, we will recall the proof for the sake of completeness.
\begin{prop} \label{prop:dividing_func}
Let $f$ be a dividing function on $X$, then it has only simple poles and zeroes and its residues at the poles, with respect to a real local coordinate with positive orientation, are negative. Conversely, if $X$ is dividing and $f$ is a real meromorphic function on $X$ with simple real poles and negative residues with respect to positive real local coordinate, then $f$ is dividing.
\end{prop}
\begin{proof}
Let $p$ be a zero of $f$, then $p \in X(\R)$. Let $t$ be a real local coordinate centered at $p$. We note that if we have a zero of higher order that $f(t) = a t^k + \ldots$, hence if $t$ is small enough it can not preserve the part of the disc with positive imaginary part, unless $k=1$. Since if $f$ is dividing then so is $-1/f$, hence a similar conclusion applies to poles.

In order to prove the second part of the claim we fix a real positively oriented local coordinate $t$ at a pole, then:
\[
\lim_{t \to -} t f(t) = a.
\]
Since the limit exists in particular the limit exists when we approach $0$ along the positive imaginary axis. Then the imaginary part of $f(t)$ is also positive by assumption and hence the real part of $t f(t)$ is always negative, and hence so is the limit.

Conversely, assume that $f$ is a real meromorphic function on $X$ with simple real poles and negative residues with respect to positive real local coordinate. Then $\operatorname{Im}(f)$ is a harmonic function on $X_{+}$ and it vanishes at every point of the boundary, except for the poles. The above limit argument shows that in fact for a positively oriented coordinate the imaginary part of $f$ is positive on $X_{+}$ near the poles. Therefore, from the minimum and maximum principle for harmonic functions it follows that $\operatorname{Im}(f) > 0$ on $X_{+}$.
\end{proof}

\begin{rem}
Another way to see that the residues are negative is as follows. Let $p$ be a simple pole of $f$ and $t$ again a real positive local coordinate centered at $p$. Write the Laurent expansion of $f$ with respect to $t$: $f(t) = a/t + b + \ldots$. We have that if $\operatorname{Im} t > 0$, then the sign of $\operatorname{Im} (a/t)$ is the opposite of the sign of $a$, so for $t$ of very small modulus we conclude that $a$ has to be negative.
\end{rem}

Let us assume from now on that $X$ is dividing and $X(\R)$ has $k$ components, $X_0,\ldots,X_{k-1}$. We can pullback $\tau$ to an anti-holomorphic involution on $\tilde{X}$, we'll denote the pullback by $\tau$ as well. We recall the construction of a special symmetric basis for the homology of $X$ from \cite{Vin93}. We take a point $s_i$ on $X_i$ and for each $i=1,\ldots,k-1$ we take a path $C_i$ connecting $s_0$ to $s_i$ and containing no other real points. Then we set $A_{g+1-k+i} \sim X_i$ and $B_{g+1-k+i} \sim \pm (C_i - C_i^{\tau})$. Here $\sim$ stands for integral homology we choose the sign in $B_{g+1-k+i}$ so that $\langle A_{g+1-k+i},B_{g+1-k+i} \rangle = 1$, where the pairing is the intersection pairing. Then we complete this to a symmetric homology basis on $X$.

We fix a corresponding basis of holomorphic differentials $\omega_1,\ldots,\omega_g$. Then, as in \cite{Vin93}, we have that $\overline{\tau^* \omega_j} = \omega_j$. Recall from \cite[Ch.\ 3]{Vin93} that the Jacobian variety of $X$ has several real sub-tori, associated each to a different choice of signs $(v_0,\ldots,v_{g-r})$, where $r = g + 1 - k$, defined by:
\begin{multline*}
T_v = \{ \zeta \in J(X) \mid \zeta = \frac{v_1}{2} e_{r+1} + \ldots + \frac{v_{g-r}}{2} e_g + a_1 (\Omega_1 - \frac{e_2}{2}) + a_2 (\Omega_2 - \frac{e_1}{2}) + \ldots + \\ a_{r-1} (\Omega_{r-1} - \frac{e_r}{2}) + a_r (\Omega_r - \frac{e_{r-1}}{2}) + a_{r+1} \Omega_{r+1} + \ldots + a_g \Omega_g\}.
\end{multline*}  
Here the $e_j$ and $\Omega_j$ are columns of the identity matrix and $\Omega$, respectively.

Write $e_1,\ldots,e_g$ for the standard basis of $\Z^g$. Let us fix $\chi \in \T_{v}$, then by our assumption and \cite[Eq,\ 3.12]{Vin93}, we have that:
$$
\varphi(\chi) + \varphi(\chi^{\tau}) = \varphi(\chi) + \overline{\varphi(\chi)} = v_1 e_{r+1} + \ldots + v_{g-r} e_g.
$$
Using this fact we obtain the following lemma about the behavior of $K(\chi,\cdot,\cdot)$.
\begin{lem} \label{lem:K_bar}
For every two distinct points $p,q \in X$, we have that:
$$
\overline{K(\chi,p,q)} = - K(\chi,q^{\tau},p^{\tau}).
$$
\end{lem}
\begin{proof}
Recall that we have the following identity for theta functions:
$$
\theta{a \brack b}(z) = e^{2 \pi i \langle z + b +\frac{1}{2} \Omega a, a \rangle} \theta(z + b +\Omega a).
$$
Let us write:
$$
G(z) = \frac{\theta{a \brack b}(z)}{\theta{a \brack b}(0)} = e^{2 \pi i \langle z, a \rangle} \frac{\theta(z + b + \Omega a)}{\theta(b + \Omega a)}.
$$
Then, by \cite[Prop.\ 2.3]{Vin93}, we have that, for real $a$ and $b$, such that $b + \Omega a \in \T_v$:
$$
\overline{G(z)} = e^{- 2 \pi i \langle \bar{z}, a \rangle} \frac{\theta(\bar{z} - b - \Omega a + v_1 e_{r+1} + \ldots + v_{g-r} e_g)}{\theta(-b - \Omega a + v_1 e_{r+1} + \ldots + v_{g-r} e_g)} = G(- \bar{z}).
$$
Now note that we have:
$$
K(\chi,p,q) = \frac{G(\varphi(q) - \varphi(p))}{E_{\Delta}(q,p)}.
$$
Hence, applying the above equality and \cite[Eq.\ 2.12]{Vin93}, we get:
$$
\overline{K(\chi,p,q)} = \frac{G(\varphi(p^{\tau}) - \varphi(q^{\tau}))}{\overline{E_{\Delta}(q,p)}}.
$$
By the fact that the prime form is anti-symmetric, we get that to prove the result we need only to show that:
$$
\overline{E(p,q)} = E(p^{\tau},q^{\tau}).
$$
Now by \cite[Eq.\ 19]{Fay73} we have that:
$$
E(p,q) = \frac{\theta[\varphi(\Delta)](q-p)}{h(p) h(q)}.
$$
Here $h$ is a holomorphic section of $\Delta$, satisfying $h^2(p) = \sum_{j=1}^g \dfrac{\partial \theta[\varphi(\Delta)]}{\partial z_j}(0) \omega_j(p)$. By \cite[Prop.\ 6.11]{Fay73} we have that there exists an open cover of $X$, trivializing $\Delta$, such that $h$ is real and positively oriented,  Now applying again \cite[Prop.\ 2.3]{Vin93} we get that: $\overline{h^2(p)} = c h^2(p^{\tau})$. Therefore we get the desired result. See also \cite[Cor.\ 6.12]{Fay73}.
\end{proof}

The following fact was essentially proved in the proof of \cite[Thm.\ 2.1]{AlpVin02}, we recall the proof to make the exposition more self-contained.
\begin{cor} \label{cor:u_adjointness}
Let $D = \sum_{j=1}^m p_j$ be an effective reduced divisor on $X$, such that $p_j \in X(\R)$ for every $j=1,\ldots,m$. Then if $\chi \in \T_v$, we have that: $\buc(p)^* = - J \bucl(p^{\tau})$, where $J$ is a signature matrix that depends on $v$. In particular if $v = 0$, then $J = I$, the identity matrix.
\end{cor}
\begin{proof}
Note that $\buc(p)^* = \begin{pmatrix}
\frac{\overline{K(\chi,\tilde{p}_1,p)}}{\sqrt{dt_1}(\tilde{p_1}}, \ldots,\frac{\overline{K(\chi,\tilde{p}_m,p)}}{\sqrt{dt_m}(\tilde{p_m}}
\end{pmatrix}$. Now applying Lemma \ref{lem:K_bar} we get that:
\[
\frac{\overline{K(\chi,\tilde{p}_j,p)}}{\sqrt{dt_j}(\tilde{p_j})} == - \frac{K(\chi,p^{\tau},\tilde{p}_j^{\tau})}{\sqrt{dt_j}(\tilde{p_j}} = - \frac{K(\chi,p^{\tau},T_j \tilde{p}_j)}{\sqrt{dt_j}(\tilde{p}_j}= - a_{\chi}(T_j) \frac{K(\chi,p^{\tau},\tilde{p}_j)}{\sqrt{dt_j}(\tilde{p}_j}.
\]
Here $T_j \in \pi_1(X)$ that maps $\tilde{p}_j$ to $\tilde{p}^{\tau}_j$. Assume that $p_j \in X_s$, where $X_s$ is some components of $X(\R)$. Then by the definition of the symmetric basis in $H_1(X,\Z)$ we have that if $s=0$, then $T_j \sim 0$ or $T_j \sim B_{g+1-k+s}$ if $s = 1,\ldots,k-1$. Since $a_{\chi}$ is a unitary character it factors through $H_1(X,\Z)$. So either $a_{\chi}(T_j) = 1$, if $s=0$ or $a_{\chi}(T_j) = e^{2 \pi i b_{g+k-1+s}}$. Now by \cite[Eq.\ 3.9]{Vin93} we have that $b_{g+1-k+s} = v_{s}/2$ and we are done.
\end{proof}

\begin{prop} \label{prop:bezoutian_hermitian}
Assume that $\chi \in \T_v$ and let $D = \sum_{j=1}^m p_j$ be an effective reduced divisor with all $p_j \in X(\R)$. Let $f,g \in \cL(D)$ be real. Then under our assumptions, $B_{\chi,D}(f,g)$ is $J$-Hermitian, where $J$ is a signature matrix that depends on $v$ obtained above.
\end{prop}
\begin{proof}
Fix two distinct points $p,q \in X(\R)$ not on $D$, then by Proposition \ref{lem:K_bar} we have that:
$$
\overline{b_{\chi}(f,g)(p,q)} = b_{\chi}(f,g)(q,p).
$$
Now applying Proposition \ref{prop:bezoutians_formula} and Lemma \ref{lem:K_bar}, we get that:
$$
\sum_{i,j =1}^m \overline{b_{ij}} \frac{\overline{K(\chi,p,\tilde{p}_i)} \overline{K(\chi,\tilde{p}_j,q)}}{\overline{\sqrt{dt_i}(\tilde{p}_i)\sqrt{dt_j}(\tilde{p}_j)}} = \sum_{i,j =1}^m \overline{b_{ij}} a_{\chi}(T_i) a_{\chi}(T_j)\frac{K(\chi,\tilde{p}_i,p) K(\chi,q,\tilde{p}_j)}{\sqrt{dt_i}(\tilde{p}_i)\sqrt{dt_j}(\tilde{p}_j)}.
$$
Hence comparing coefficients we get that, $b_{ij} = a_{\chi}(T_i) a_{\chi}(T_j) \overline{b_{ji}}$. Conclude that:
\[
B_{\chi,D}(f,g) = J B_{\chi,D}(f,g)^* J.
\]
\end{proof}

Assume that $\chi \in \T_0$ and let $D$ be a divisor as in Proposition \ref{prop:bezoutian_hermitian} above. Let us assume that we have two functions $f,g \in \cL(D)$ are real and $f/g$ is dividing. Let $(f)_{\infty}$ and $(g)_{\infty}$ be the divisors of poles of $f$ and $g$, respectively. Let us assume that $D = (f)_{\infty} \vee (g)_{\infty}$ the supremum of divisors of $f$ and $g$. We can replace both $f$ and $g$ by real linear combinations so that $D= (f)_{\infty} = (g)_{\infty}$. If we have a matrix $\left(\begin{smallmatrix} a & b \\ c & d \end{smallmatrix}\right) \in \SL_2(\R)$, then $h = \frac{ a f + b g}{c f + d g} = \frac{ a + b (f/g)}{c + d (f/g)}$. Hence for every $p \in X$ not a zero of $g$ if $h(p) \in \R$, then $(f/g)(p)$ in $\R$ and hence $p \in X(\R)$, so $h$ is dividing as well.

The poles of $f/g$ are thus at real zeroes of $g$. Now if $p$ is a complex zero of $g$ then $f$ also has a zero at $p$ and thus $B_{\chi,D}(f,g) \buc(\tilde{p}) = 0$ by Corollary \ref{cor:common_zero}. If $p$ is a real zero of $g$, then either it is also a zero of $f$ or it is a pole of $f/g$. In the first case we apply Corollary \ref{cor:common_zero} again to get that $B_{\chi,D}(f,g) \buc(\tilde{p}) = 0$ as well. In the second case we fix a real positive coordinate $t$ centered at $p$ and applying Proposition \ref{prop:bezoutians_formula} we get:
\[
\bucl(\tilde{p}) B_{\chi,D}(f,g) \buc(\tilde{p}) = f(p) g^{\prime}(p).
\]
Note that in this case the zero of $g$ has to be simple, since every pole of $f/g$ is simple. Using Corollary \ref{cor:u_adjointness} and the fact that $\chi \in \T_0$ we conclude that:
\[
\langle B_{\chi,D}(f,g) \buc(\tilde{p}), \buc(\tilde{p}) \rangle = - f(p) g^{\prime}(p).
\]
Note that the residue of $f/g$ at $p$ is $f(p)/g^{\prime}(p) < 0$ and deduce that $- f(p) g^{\prime}(p) > 0$. This leads us to the following proposition.

\begin{prop} \label{prop:bezout_definite}
Assume that $\chi \in \T_0$ and let $D$ be a divisor as in Proposition \ref{prop:bezoutian_hermitian} above and assume that $f,g \in \cL(D)$ are real and that $f/g$ is dividing. If $B_{\chi,D}(f,g)$ is invertible then $B_{\chi,D}(f,g) \geq 0$. 
\end{prop}
\begin{proof}
 We first reduce to the case that $D = (f)_{\infty} = (g)_{\infty}$. We know that $D \geq (f)_{\infty} \vee (g)_{\infty}$ and for every point $p_j$ that is neither a pole of $f$ nor $g$, the $j$-th column and row of $B_{\chi,D}(f,g)$ are zero and this contradicts our assumption.

So as in the preceding discussion we can assume that $(f)_{\infty} = (g)_{\infty} = D$. Since $B_{\chi,D}(f,g)$ is invertible, we get that all the zeroes of $g$ are simple and distinct from the zeroes of $f$. Let $q_1,\ldots,q_m$ be the zeroes of $g$ and fix a lift $\tilde{q}_j \in \tilde{X}$. By Corollary \ref{cor:u_independence_divisor} we know that $\buc(\tilde{q_j})$ are linearly independent and by the discussion above they are orthogonal with respect to the bilinear form defined by $B_{\chi,D}(f,g)$. Furthermore the discussion above combined with \cite[Prop.\ 2.2.3]{GLR} gives us that $B_{\chi,D}(f,g)$ is positive definite.
\end{proof}

\begin{rem}
In fact the assumption of $B_{\chi,D}(f,g)$ being invertible can be relaxed, if we assume that $g$ has simple zeroes and $(f/g)$ is dividing, it will still follow that $B_{\chi,D}(f,g) \geq 0$. Indeed the assumptions imply that we are allowing $f$ and $g$ to have common zeroes. The vectors $\buc(\tilde{q}_j)$ are still linearly independent, however some them may be isotropic vectors of $B_{\chi,D}(f,g)$. Looking at those vectors that are not isotropic, we can still deduce that $B_{\chi,D}(f,g) \geq 0$.
\end{rem}

\section{Livsic-type Determinantal Representations of Curves} \label{sec:det_rep_curve}

We shall first fix some notations to be used constantly from now on. Let $C \hookrightarrow \pp^d$ be a projective curve of degree $n$ not contained in any hypersurface. Let $X$ be the normalizing Riemann surface of $C$. Let $\iota \colon X \to \pp^d$ be the composition of the normalization map with the embedding of $C$. Let us assume that $C$ intersects the hyperplane at infinity at $n$ distinct non-singular points. Otherwise we apply a linear transformation to achieve it. Let $\cL = \iota^* \cO(1)$ be a line bundle on $X$. Then we have global sections $\mu_0,\ldots,\mu_d \in H^0(X,\cL)$, such that $\iota(p) = (\mu_0(p):\cdots : \mu_d(p))$. We denote  $\lambda_j = \mu_j/\mu_0$, for $j=1,\ldots,d$ and set $\lambda_0 = 1$. Again applying a linear transformation if necessary we may assume that $\mu_1$ and $\mu_0$, have no common zeroes.

Fix a flat unitary line bundle $\chi$ on $X$ and a line bundle of half-order differentials, $\Delta$. We define a tensor $\gamma \in \wedge^2 \C^{d+1} \otimes M_n(\C)$ by setting $\gamma_{ij} = B_{\chi,D}(\lambda_i,\lambda_j)$, where $D$ is the divisor of zeroes of $\mu_0$. Note that by assumption the zeroes of $\mu_0$ are simple and hence for every $j=0,\ldots,d$ we have that $\lambda_j \cL(D)$. Furthermore the divisor $D$ is the divisor of poles of $\lambda_1$. In particular if $e_0,\ldots,e_d$ are the standard basis of $\C^{d+1}$, then $\gamma = \sum_{0 \leq i < j \leq d} \gamma_{ij} \otimes e_i \wedge e_j$.

Let $V \subset \pp^d$ be a linear subspace of dimension $d-2$. Writing out $\gamma(V)$, we get:
$$
\gamma(V) = \sum_{0 \leq i < j \leq d} (a_{i0} a_{j1} - a_{j0} a_{i1}) \gamma_{ij}. 
$$
By the properties of the Bezoutian, we get that:
$$
\gamma(V) = \sum_{0 \leq i < j \leq d} B( a_{i0} \lambda_i + a_{j0} \lambda_j, a_{i1} \lambda_i + a_{j1} \lambda_j).
$$
Now rearranging the terms and using the linearity and the fact that $B(f,f) = 0$, for every meromorphic function $f$, one gets that:
$$
\gamma(V) = B(\sum_{i=0}^d a_{i0} \lambda_i, \sum_{j=0}^d a_{j1} \lambda_j).
$$
So we get the following:
\begin{lem} \label{lem:bezoutian}
Let $C$, $X$ and $V$ as above, then there exist linear combinations of the $\lambda_j$, namely $\kappa_0 = \sum_{i=0}^d a_{i0} \lambda_i$ and $\kappa_1 =\sum_{j=0}^d a_{j1} \lambda_j$, such that:
$$
\gamma(V) = B(\kappa_0,\kappa_1).
$$
\end{lem}

The main result of this section is the following theorem:
\begin{thm} \label{thm:vr_rep_curves}
The curve $C$ admits a \vr{} Livsic-type determinantal representation $\gamma$.
\end{thm}
\begin{proof}
By Corollary \ref{cor:det_vanish} we have that for every $1 \leq i < j \leq d$ and every affine point of $C$ we have that:
$$
\left( \lambda_i(p) \gamma_{0j} - \lambda_j(p) \gamma_{0i} + \gamma_{ij} \right) \buc(p) = 0.
$$
Let us write for $0 \leq i < j < k \leq d$, $L_{ijk} = \lambda_k \gamma_{ij} - \lambda_j \gamma_{ik} + \lambda_i \gamma_{jk}$. Then one notes that for every $1 \leq i < j < k \leq d$, we have that:
$$
L_{ijk} = \lambda_k L_{0ij} - \lambda_j L_{0ik} + \lambda_i L_{0jk}.
$$
Hence for every $p \in X$ not a pole of the $\lambda_j$ we have that $\iota(p) \in D(\gamma)$. Hence $C \subset D(\gamma)$, since $C$ is the Zariski closure of its affine part and $D(\gamma)$ is closed.

Now by Corollary \ref{cor:u_independence_divisor} we have that for a generic hypersurface of the form $\mu_1 = z$, intersecting $C$ in $n$-distinct affine points, $q_1,\ldots,q_n$, the vectors $\buc(q_j)$ are a basis for $\C^n$. Hence we have that generically the kernel of $\gamma$ is one-dimensional. Now $\deg \gamma \leq n$ on the other hand since $C \subset D(\gamma)$ we have that $\deg \gamma \geq n$, conclude that $\deg \gamma = n$ and therefore, $\gamma$ is \vr{}. Furthermore this implies that $D(\gamma)$ is of pure dimension $1$ and of degree $n$. Conclude that $D(\gamma) = C$ as sets.
\end{proof}

\begin{rem}
Note that if pull back $\cK$, the kernel sheaf of the determinantal representation, to the normalization and mod out torsion we will get $\chi \otimes \Delta$ (up to a twist).
\end{rem}

We will finish this section with some examples of the construction in the case of genus $0$ curves.

\begin{ex}
Using the methods of \cite[Ch.\ 9]{LKMV} we get that the following matrices are a realization of the twisted cubic curve:
\begin{align*}
& \gamma_{01} = \begin{pmatrix} 1 & 0 & 0 \\ 0 & 0 & 0 \\ 0 & 0 & 0 \end{pmatrix} \,,\, \gamma_{02} = \begin{pmatrix}
0 & 1 & 0 \\ 1 & 0 & 0 \\ 0 & 0 & 0 \end{pmatrix} \,,\, \gamma_{03} = \begin{pmatrix} 0 & 0 & 1 \\ 0 & 1 & 0 \\ 1 & 0 & 0\end{pmatrix}, \\
& \gamma_{12} = \begin{pmatrix} 0 & 0 & 0 \\ 0 & 1 & 0 \\ 0 & 0 & 0 \end{pmatrix} \,,\, \gamma_{13} = \begin{pmatrix} 0 & 0 & 0 \\ 0 & 0 & 1 \\ 0 & 1 & 0\end{pmatrix} \,,\, \gamma_{23} = \begin{pmatrix} 0 & 0 & 0 \\ 0 & 0 & 0 \\ 0 & 0 & 1 \end{pmatrix}.
\end{align*}
\end{ex}

\begin{ex}
Similarly one obtain for a cuspidal plane monomial quintic the Livsic-type determinantal representation:
\begin{align*}
& \gamma_{01} = \begin{pmatrix}
0 & 0 & 1 & 0 & 0 \\ 0 & 1 & 0 & 0 & 0 \\ 1 & 0 & 0 & 0 & 0 \\ 0 & 0 & 0 & 0 & 0 \\ 0 & 0 & 0 & 0 & 0
\end{pmatrix} \,,\, \gamma_{02} = \begin{pmatrix}
0 & 0 & 0 & 1 & 0 \\ 0 & 0 & 1 & 0 & 0 \\ 0 & 1 & 0 & 0 & 0 \\ 1 & 0 & 0 & 0 & 0 \\ 0 & 0 & 0 & 0 & 0 
\end{pmatrix} \,,\, \gamma_{03} = \begin{pmatrix}
0 & 0 & 0 & 0 & 1 \\ 0 & 0 & 0 & 1 & 0 \\ 0 & 0 & 1 & 0 & 0 \\ 0 & 1 & 0 & 0 & 0 \\ 1 & 0 & 0 & 0 & 0
\end{pmatrix}, \\
& \gamma_{12} = \begin{pmatrix}
0 & 0 & 0 & 0 & 0 \\ 0 & 0 & 0 & 0 & 0 \\ 0 & 0 & 0 & 0 & 0 \\ 0 & 0 & 0 & -1 & 0 \\ 0 & 0 & 0 & 0 & 0
\end{pmatrix} \,,\, \gamma_{13} = \begin{pmatrix}
0 & 0 & 0 & 0 & 0 \\ 0 & 0 & 0 & 0 & 0 \\ 0 & 0 & 0 & 0 & 0 \\ 0 & 0 & 0 & 0 & -1 \\ 0 & 0 & 0 & -1 & 0
\end{pmatrix} \,,\, \gamma_{23} = \begin{pmatrix}
0 & 0 & 0 & 0 & 0 \\ 0 & 0 & 0 & 0 & 0 \\ 0 & 0 & 0 & 0 & 0 \\ 0 & 0 & 0 & 0 & 0 \\ 0 & 0 & 0 & 0 & -1
\end{pmatrix}.
\end{align*}
However the scheme has an embedded point at the singularity. The affine primary decomposition is given by:
\begin{equation*}
I = (y^2-xz, x^2 y - z^2, x^3 - yz) \cap (z, y^3, x y^2, x^3 y, x^4).
\end{equation*}
The last ideal is $(x,y,z)$-primary and hence $(x,y,z)$ is an embedded prime.
\end{ex}

\begin{ex}
The following is an example of a smooth, but not projectively normal rational curve in $\pp^3$ obtained via the map $(1,t,t^2,t^3)$.
\begin{align*}
& \gamma_{01} = \begin{pmatrix}
1 & 0 & 0 & 0 \\ 0 & 0 & 0 & 0 \\ 0 & 0 & 0 & 0 \\ 0 & 0 & 0 & 0 \end{pmatrix} \,,\, \gamma_{02} = \begin{pmatrix}
0 & 0 & 1 & 0 \\ 0 & 1 & 0 & 0 \\ 1 & 0 & 0 & 0 \\ 0 & 0 & 0 & 0 \end{pmatrix} \,,\, \gamma_{03} = \begin{pmatrix}
0 & 0 & 0 & 1 \\ 0 & 0 & 1 & 0 \\ 0 & 1 & 0 & 0 \\ 1 & 0 & 0 & 0 \end{pmatrix}, \\
& \gamma_{12} = \begin{pmatrix}
0 & 0 & 0 & 0 \\ 0 & 0 & 1 & 0 \\ 0 & 1 & 0 & 0 \\ 0 & 0 & 0 & 0 \end{pmatrix} \,,\, \gamma_{13} = \begin{pmatrix}
0 & 0 & 0 & 0 \\ 0 & 0 & 0 & 1 \\ 0 & 0 & 1 & 0 \\ 0 & 1 & 0 & 0 \end{pmatrix} \,,\, \gamma_{23} = \begin{pmatrix}
0 & 0 & 0 & 0 \\ 0 & 0 & 0 & 0 \\ 0 & 0 & 0 & 0 \\ 0 & 0 & 0 & 1 \end{pmatrix}.
\end{align*}
\end{ex}

\section{Hyperbolic Curves in $\pp^d$} \label{sec:hyper_curve}

From now let us assume that $C$ is a real curve, then the involution on $\pp^d$ obtained from complex conjugation of coordinates, induces an anti-holomorphic involution on $X$. Note that $\dim H^0(X,\cL) \geq d+1$, in particular if $W \subset H^0(X,\cL)$ is the subspace spanned by the real sections $\mu_0,\ldots,\mu_d$, then in fact $\iota$ is a map from $X$ to $\pp W^*$. We identify $\pp W^*$ with $\pp W$, by setting the basis $\mu_0,\ldots,\mu_d$ to be orthonormal. Furthermore a section $\nu \in W$ is real if and only if it is a linear combination of the $\mu_j$ with real coefficients. Let us assume that there exists a real linear subspace $V \subset \pp W$ of dimension $d-2$, such that $C$ is hyperbolic with respect to $V$, then we have that:

\begin{lem} \label{lem:hyper_dividing}
There exist real $\nu_0, \nu_1 \in H^0(\tilde{X},\cL)$, such that the meromorphic function $\lambda$, on $X$, defined by $\lambda = \nu_1/\nu_0$ is dividing. In particular $X$ is dividing.
\end{lem}
\begin{proof}
Consider $V \subset H^0(X,\cL)$ and assume at first that $V$ is spanned by $\mu_2,\ldots,\mu_d$. Then every real hypersurface containing $V$ is spanned by $s \mu_0 + t \mu_1$ and $V$, where $s,t \in \R$ not both zero. Set $\lambda = \mu_1/\mu_0$. Clearly if $p \in X(\R)$, then $\lambda(p)$ is real. On the other hand if $\lambda(p)$ is real then either $\lambda(p) = \alpha \in \R$ or $\lambda(p) = \infty$, then take the hyperplane $H$ spanned by $\mu_0 + \alpha \mu_1$ and $V$ in the first case and $mu_1$ and $V$ in the second case. Observe that by hyperbolicity $H \cap C \subset C(\R)$. Now $\iota(p) = (\mu_0(p),\ldots,\mu_d(p))$, in particular $\iota(p) \in H \cap C$ and hence is real and therefore $p \in X(\R)$.

Now if $V$ is spanned by real sections $\nu_2,\ldots,\nu_d$. We can complete this set to a a real basis of $H^0(X,\cL)$, by adding two more sections $\nu_0^{\prime}$ and $\nu_1^{\prime}$. Now since hyperbolicity is invariant under real coordinate changes, we see that we get the required $\nu_0$ and $\nu_1$ by pulling back $\mu_0$ and $\mu_1$.
\end{proof}

This discussion leads us to the main result of this section.

\begin{thm} \label{thm:hyp-liv-det-rep}
The curve $C$ admits \vr{} Hermitian Livsic type determinantal representations $\gamma$, parametrized by flat unitary line bundles in $\T_0$, a real subtorus of the Jacobian variety of the desingularizing Riemann surface of $C$, such that for every real $d-2$-dimensional linear subspace $U \subset \pp^d$, we have that $\gamma(U)$ is definite if and only if $C$ is hyperbolic with respect to $U$. In particular if the line bundle is two-torsion, then the resulting determinantal representation is real symmetric.
\end{thm}
\begin{proof}
Fix $\chi$ a flat unitary line bundle on $\T_0$, by \cite[Cor.\ 4.3]{Vin93} $h^0(\chi \otimes \Delta) = 0$. Then by Proposition \ref{prop:bezoutian_hermitian}, we have that each $\gamma_{ij}$ is Hermitian, since the $\lambda_j$ are real functions with real poles.

It suffices to prove that if $C$ is hyperbolic with respect to $U$ then $\gamma(U)$ is definite, since the converse has already been proved. However, we note that by Lemma \ref{lem:bezoutian} we have that for every such $U$ there exist two functions $\kappa_0$ and $\kappa_1$, such that by Lemma \ref{lem:hyper_dividing} $\lambda = \kappa_1/\kappa_0$ is dividing. Since $\gamma$ is \vr{} we note that $U \cap C = \emptyset$ implies that $\gamma(U)$ is invertible. Now we note that $\gamma(U) = B_{\chi,D}(\kappa_0,\kappa_1)$ and hence we can apply Proposition \ref{prop:bezout_definite} to get the result. To get a real symmetric representation apply Propostion \ref{prop:bezotian_symmetric}.
\end{proof}

If $C$ admits a \vr{} Hermitian Livsic type determinantal representation $\gamma$, then so does the hypersurface, $Y$, corresponding to it via the incidence correspondence. Now we can lift the determinantal representation to some hypersurface, $Y^{\prime}(\gamma) \subset \pp^N$. If $C$ is hyperbolic with respect to some real $d-2$-dimensional linear subspace $V \subset \pp^d$, then by the above theorem $\gamma(V)$ is definite; we conclude that $Y^{\prime}(\gamma)$ is hyperbolic with respect to $V$. On the other hand if $Y^{\prime}(\gamma)$ is hyperbolic with respect to $V$ then so is $Y$, and thus by Proposition \ref{prop:hyperbolic-equiv} $C$ is hyperbolic with respect to $V$; applying the theorem again we see that $\gamma(V)$ is definite. Therefore we obtain:

\begin{cor}
Set $\cH(Y) = \left\{ V \in \Gr(d-2,d) \mid Y \mbox{ is hyperbolic w.r.t. } V \right\}$. and similarly $\cH(Y^{\prime}) = \left\{ U \in \pp^N \mid Y^{\prime} \mbox{ is hyperbolic w.r.t. } U \right\}$, then:
$$
\cH(Y) = \Gr(d-2,d) \cap \cH(Y^{\prime}).
$$
In particular the cone over $\cH(Y)$ is a disjoint union of two extendably convex connected components in the sense of Buseman.
\end{cor}

\nocite{M2}
\bibliographystyle{plain}
\bibliography{Hyperbolic_Curves_in_Pn}

\end{document}